\documentclass{amsart}
\usepackage{amssymb}
\usepackage{mathrsfs}
\usepackage{stmaryrd}
\usepackage{bbm}
\usepackage{oldgerm}
\usepackage[english]{babel}
\usepackage[T1]{fontenc}
\usepackage[latin1]{inputenc}
\usepackage[all]{xy}
\usepackage{mathabx}
\usepackage{hyperref}

\newtheorem{theorem}[subsection]{Theorem}
\newtheorem{proposition}[subsection]{Proposition}
\newtheorem{corollary}[subsection]{Corollary}
\newtheorem{lemma}[subsection]{Lemma}

\theoremstyle{definition}

\newtheorem{definition}[subsection]{Definition}
\newtheorem{remark}[subsection]{Remark}

\numberwithin{equation}{subsection}

\def\L{\Lambda}
\def\G{\Gamma}
\def\g{\gamma}
\def\d{\delta}
\def\D{\Delta}
\def\e{\epsilon}

\def\a{\alpha}
\def\r{\rho}
\def\O{\Omega}

\def\i{\iota}

\newcommand{\dctf}{D_{\mathrm{ctf}}^b}
\newcommand{\ra}{\rightarrow}
\newcommand{\xra}{\xrightarrow}

\newcommand{\isora}{\xrightarrow{\sim}}
\newcommand{\R}{\mathrm R}
\newcommand{\h}{\mathrm H}
\newcommand{\CH}{\mathrm{CH}}
\newcommand{\iso}{\cong}
\newcommand{\cl}{\mathrm{cl}}
\newcommand{\bQ}{\mathbb{Q}}
\newcommand{\bF}{\mathbb{F}}
\newcommand{\bD}{\mathbf{D}}
\newcommand{\bZ}{\mathbb{Z}}
\newcommand{\A}{\mathbb{A}}
\newcommand{\ddy}{\geqslant}
\newcommand{\xdy}{\leqslant}
\newcommand{\sK}{\mathscr{K}}
\newcommand{\sF}{\mathscr{F}}
\newcommand{\sG}{\mathscr{G}}
\newcommand{\sH}{\mathscr{H}}
\newcommand{\sO}{\mathscr{O}}

\newcommand{\sL}{\mathscr{L}}
\newcommand{\fF}{\mathfrak{F}}
\newcommand{\shom}{\mathscr{H}\!om}
\newcommand{\send}{\mathscr{E}nd}
\newcommand{\ot}{\otimes}
\newcommand{\ti}{\times}
\newcommand{\bxt}{\boxtimes}
\newcommand{\xx}{X\ti_kX}
\newcommand{\uu}{U\ti_kU}
\newcommand{\ev}{\mathrm{ev}}
\newcommand{\ol}{\overline}
\newcommand{\pr}{\mathrm{pr}}
\newcommand{\id}{\mathrm{id}}
\newcommand{\rshom}{\R\shom}

\newcommand{\yz}{Y\ti_kZ}
\newcommand{\yzlb}{(Y\ti_kZ)'}

\newcommand{\yzlp}{Y\Asterisk_kZ}
\newcommand{\yzd}{(Y\ti_kZ)^{\dagger}}
\newcommand{\sI}{\mathscr{I}}

\newcommand{\xxlp}{X\Asterisk_kX}
\newcommand{\xxlb}{(\xx)'}
\newcommand{\xxd}{(X\ti_kX)^{\dagger}}
\newcommand{\wt}{\widetilde}
\newcommand{\rmt}{\mathrm{t}}
\newcommand{\rw}{\mathrm{w}}
\newcommand{\jr}{j^{(R)}}
\newcommand{\fr}{f^{(R)}}
\newcommand{\xr}{X^{(R)}}
\newcommand{\sHr}{\sH^{(R)}}
\newcommand{\xxr}{(\xxlp)^{(R)}}
\newcommand{\sE}{\mathscr{E}}

\newcommand{\bp}{\bigoplus}
\newcommand{\ma}{\mapsto}
\newcommand{\vx}{V\ti_k X}
\newcommand{\uv}{U\ti_kV}
\newcommand{\vv}{V\ti_kV}

\newcommand{\vvlb}{(V\ti_k V)'}

\newcommand{\vvlp}{V\Asterisk_k V}
\newcommand{\vvd}{(V\ti_kV)^{\dagger}}

\newcommand{\gk}{G_K}
\newcommand{\gkl}{G_{K,\log}}
\newcommand{\grl}{\mathrm{Gr}^r_{\log}{G_K}}

\newcommand{\OX}{\Omega^1_{X/k}}

\newcommand{\sr}{S^{(R)}}
\newcommand{\adj}{\mathrm{adj}}

\DeclareMathOperator{\End}{End}
\DeclareMathOperator{\spec}{Spec}
\DeclareMathOperator{\Hom}{Hom}
\DeclareMathOperator{\aut}{Aut}
\DeclareMathOperator{\rk}{rk}

\begin{document}
  \title{Refined characteristic class and Conductor formula}
  \author{Haoyu Hu}
  \address{Haoyu Hu, IHES, Le Bois-Marie, 35 Rue de Chartres, 91440 Bures-sur-Yvette, France.}
  \email{ haoyu@ihes.fr \& huhaoyu@mail.nankai.edu.cn}
  \begin{abstract}
In this article, we prove a conductor formula in a geometric situation that generalizes the Grothendieck-Ogg-Shafarevich formula. Our approach uses the ramification theory of Abbes and Saito, and relies on Tsushima's refined characteristic class.
  \end{abstract}
  \maketitle
  \tableofcontents
  \section{Introduction}

  \subsection{}
This article is devoted to the proof of a conductor formula for $\ell$-adic sheaves in a geometric situation \eqref{c f intro form} which generalizes the classical Grothendieck-Ogg-Shafarevich formula (\cite{sga5} X 7.1) as well as the index formula of Saito (\cite{saito cc} 3.8). It uses the ramification theory developed by Abbes and Saito and it relies on a previous work of Tsushima, who proved a special case (\cite{ts2} 5.9).

\subsection{}\label{setting cond form intro}
  Let $k$ be a perfect field of characteristic $p>0$, $f:X\ra Y$ a proper flat morphism of smooth connected $k$-schemes and $d$ the dimension of $X$. We assume that $\dim Y=1$ and let $y$ be a closed point of $Y$, $\ol y$ a geometric point localized at $y$, $Y_{(\ol y)}$ the strict localization of $Y$ at $\ol y$ and $\ol\eta$ a geometric generic point of $Y_{(\ol y)}$. Put $W=Y-\{y\}$, $V=f^{-1}(W)$ and that $Q=f^{-1}(y)$. We assume that the canonical projection $f_V:V\ra W$ is smooth and $Q$ is a divisor with normal crossing on $X$. Let $D$ be a divisor with simple normal crossing on $X$ containing $S=Q_{\mathrm{red}}$ such that $D\cap V$ is a divisor with simple normal crossing relatively to $W$. We put $U=X-D$ and let $j:U\ra X$ be the canonical injection. We consider the diagram
  \begin{equation*}
  \xymatrix{\relax
  U\ar[r]^-{\nu}\ar[rd]_{f_U}&V\ar@{}|{\Box}[rd]\ar[r]
  \ar[d]^{f_V}&X\ar@{}|{\Box}[rd]\ar[d]^{f}&Q\ar[l]\ar[d]\\
  & W\ar[r]&Y&y\ar[l]}
\end{equation*}
where $\nu$ is the canonical injection and $f_U=f_V\circ \nu$.
We fix a prime number $\ell$ invertible in $k$, and an Artinian local $\mathbb{Z}_{\ell}$-algebra $\L$. Let $\sF$ be a locally constant and constructible sheaf of free $\L$-modules on $U$ such that
\begin{itemize}
  \item[(i)]
  $\sF$ is tamely ramified along the divisor $D\cap V$ relatively to $V$;
  \item[(ii)]
  the conductor $R$ of $\sF$ is effective with support contained in $S$ (\cite{rc} 8.10) and $\sF$ is isoclinic and clean along $D$ (\cite{rc} 8.22 and 8.23).
\end{itemize}

 Condition (i) implies that $f_V$ is universally locally acyclic relatively to $\nu_!(\sF)$ (\cite{sga4 1/2} Appendice to Th. Finitude, \cite{wrcb} 3.14). Since $f_V$ is proper, all cohomology groups of $\R f_{U!}(\sF)$ are locally constant and constructible on $W$. We put (\cite{sga4 1/2} Rapport 4.4)
\begin{eqnarray}
\rk_{\L}(\R\G_c(U_{\ol\eta},\sF|_{U_{\ol\eta}}))&=&\mathrm{Tr}(\id;\R\G_c(U_{\ol\eta},\sF|_{U_{\ol\eta}})),\nonumber\\
\mathrm{sw}_y(\R\G_c(U_{\ol\eta},\sF|_{U_{\ol\eta}}))&=&
\sum_{q\in\bZ}(-1)^q\mathrm{sw}_y(\R^q\G_c(U_{\ol\eta},\sF|_{U_{\ol\eta}})),\nonumber\\
\mathrm{dimtot}_y(\R\G_c(U_{\ol\eta},\sF|_{U_{\ol\eta}}))&=&
\rk_{\L}(\R\G_c(U_{\ol\eta},\sF|_{U_{\ol\eta}}))+\mathrm{sw}_y(\R\G_c(U_{\ol\eta},\sF|_{U_{\ol\eta}})),\nonumber
\end{eqnarray}
where $\mathrm{sw}_y(\R^q\G_c(U_{\ol\eta},\sF|_{U_{\ol\eta}}))$ denotes the Swan conductor of $\R^q\G_c(U_{\ol\eta},\sF|_{U_{\ol\eta}})$ at $y$.

We denote by
\begin{equation*}
\mathrm{T}^*X(\log D)=\mathbf{V}(\O^1_{X/k}(\log D)^{\vee})
\end{equation*}
the logarithmic cotangent bundle over $X$ and by $\sigma:X\ra\mathrm{T}^*X(\log D)$ zero section. Under the conditions (i) and (ii), Abbes and Saito defined the characteristic cycles of $\sF$, denoted by $CC(\sF)$, as a $d$-cycle on $\mathrm{T}^*X(\log D)$ (\cite{rc} 1.12; \cite{saito cc} 3.6; cf. \ref{charcycle}). The vertical part $CC^*(\sF)$ of $CC(\sF)$ is a $d$-cycle on $\mathrm{T}^*X(\log D)\ti_XS$ such that
\begin{equation*}
CC(\sF)=(-1)^d(\rk_{\L}(\sF)[\sigma(X)]+CC^*(\sF)).
\end{equation*}

\begin{theorem}\label{cond form intro}
We keep the notation and assumptions of {\rm \ref{setting cond form intro}} and assume moreover that $S=D$ (i.e., $U=V$) or that $\rk_{\L}(\sF)=1$. Then, for any section $s:X\ra \mathrm{T}^*X(\log D)$, we have the following equality in $\L$
\begin{equation}\label{c f intro form}
 \mathrm{dimtot}_y(\R\G_c(U_{\ol\eta},\sF|_{U_{\ol\eta}}))
-\rk_{\L}(\sF)\cdot\mathrm{dimtot}_y(\R\G_c(U_{\ol\eta},\L))=(-1)^{d+1}\deg(CC^*(\sF)\cap[s(X)]).
\end{equation}
\end{theorem}

The case where $\rk_{\L}(\sF)=1$ is due to Tsushima (\cite{ts2} 5.9). Although we follow the same lines for sheaves of higher ranks, the situation is technically more involved. Our approach requires the assumption that $S=D$.

\subsection{}
To prove \ref{cond form intro}, we follow the strategy of Saito for the proof of an index formula for $\ell$-adic sheaves on proper smooth varieties \cite{saito cc}. The latter can be schematically divided into two steps. The first step uses the theory of cohomological correspondences due to Grothendieck and Verdier to associate a cohomology class to the $\ell$-adic sheaf, called the {\it characteristic class}, that computes its Euler-Poincar\'e characteristic by the Lefschetz-Verdier formula (\cite{sga5} III). The second step is more geometric. It consists of computing the characteristic class as an intersection product using the ramification theory developed by Abbes and Saito~\cite{as i}.

\subsection{}
The analogous approach for the proof of the conductor formula \eqref{c f intro form} was started by Tsushima in \cite{ts2}. He refined the characteristic class of an $\ell$-adic sheaf into a cohomology class with support in the wild locus, called in this article the {\it refined characteristic class}. He proved a Lefschetz-Verdier formula for this class (\cite{ts2} 5.4) which amounts to say that it commutes with proper push-forward. On a smooth curve, the refined characteristic class gives the Swan conductor (\cite{ts2} 4.1). The main goal of this article is to prove an intersection formula that computes the refined characteristic class.

\subsection{}
 More precisely, with the notation and assumptions of \ref{setting cond form intro}, the refined characteristic class $C_S(j_!(\sF))$ of $j_!(\sF)$ is defined as an element in $\h^0_S(X,\sK_X)$. The Lefschetz-Verdier formula implies the following relation
 \begin{equation*}
  \mathrm{sw}_y(\R\G_c(U_{\ol\eta},\sF|_{U_{\ol\eta}}))
-\rk_{\L}(\sF)\cdot\mathrm{sw}_y(\R\G_c(U_{\ol\eta},\L))=-f_*(C_S(j_!(\sF))-\rk_{\L}(\sF)\cdot C_S(j_!(\L_U)))
\end{equation*}
in $\h^0_{\{y\}}(Y,\sK_Y)\isora\L$, where $f_*$ in the left hand side is the proper push-forward $\h^0_S(X\sK_X)\ra\h^0_{\{y\}}(Y,\sK_Y)$ (cf. \ref{X to Y not}). Assume that $D=S$ or that $\rk_{\L}(\sF)=1$. Then, our main result is the following formula (\ref{keytheorem llcc lchern})
\begin{eqnarray*}
  &&C_{S}(j_!(\sF))-\rk_{\L}(\sF)\cdot C_{S}(j_!(\L_U))\\
  &=&(-1)^{d}\rk_{\L}(\sF)\cdot c_d\left(\OX(\log D)\ot_{\sO_X}\sO_X(R)-\OX(\log D)\right)^X_{S}\cap[X]\in \h^0_{S}(X,\sK_X),\nonumber
\end{eqnarray*}
where $c_d(-)^X_S$ is a bivariant class built of localized Chern classes (cf. \ref{locchern}). The right hand side is the image of a zero cycle class in $\CH_0(S)$, whose degree is $(-1)^d\deg(CC^*(\sF)\cap[s(X)])$ (cf. \ref{pre-cond form}), which implies theorem \ref{cond form intro}.

\subsection{}
Beyond Tsushima's work already mentioned, there have been several works on the conductor formula. Abbes gave a conductor formula for an $\ell$-adic sheaf on an arithmetic surface, under the condition that the sheaf has no fierce ramification \cite{a gos}. Vidal proved that the alternating sum of the Swan conductor of the cohomology groups with compact support of an $\ell$-adic sheaf on a normal scheme over a local field only depends on its rank and its wild ramification \cite{vidal}. For an $\ell$-adic sheaf on a smooth scheme over a local field of mixed characteristic, Kato and Saito defined its Swan class, which is a $0$-cycle class supported on the wild locus, that computes the Swan conductor of the cohomology groups with compact support \cite{ksihes}. In a recent work \cite{ccen}, Saito defined the characteristic cycle of an $\ell$-adic sheaf on a smooth surface as a cycle on the cotangent bundle without the cleanliness condition. When the surface is fibered over a smooth curve, he proved a conductor formula conjectured by Deligne (\cite{ccen} 3.16).

\subsection{}
This article is organized as follows. After preliminaries on \'etale cohomology, we briefly introduce the cohomological correspondences and the characteristic class of an $\ell$-adic sheaf in $\S 4$. We recall Abbes and Saito's ramification theory in $\S 5$ and review the definition of clean sheaves and the characteristic cycle in $\S 6$. We give the definition of Tsushima's refined characteristic class and introduce the corresponding Lefschetz-Verdier formula in $\S 7$. The last section is devoted to the proof of the conductor formula.

\subsection*{Acknowledgement}
This article is a part of the author's thesis at Universit\'e Paris-Sud and Nankai University. The author would like to express his deepest gratitude to his supervisors  Ahmed Abbes and Lei Fu for leading him to this area and for patiently guiding him in solving this problem. The author would also like to thank professor Takeshi Saito for his stimulating suggestions toward to this article.
This work is developed during a long visit to IHES supported by Fonds Chern and Fondation Math\'ematiques Jacques Hadamard. The author is grateful to these institutions for their support.

\section{Notation}

\subsection{}\label{fixnotation}
In this article, $k$ denotes a perfect field of characteristic $p>0$. We fix a prime number $\ell$ invertible in $k$, an Artinian local $\mathbb{Z}_{\ell}$-algebra $\L$ and a non-trivial additive character $\psi:\mathbb{F}_p\ra \L^{\ti}$. All $k$-schemes are assumed to be separated and of finite type over $\spec (k)$.

\subsection{}\label{notsheaf}
For a $k$-scheme $X$, we denote by $D(X,\L)$ the derived category of complexes of \'etale sheaves of $\L$-modules on $X$ and by $\dctf(X,\L)$ (resp. $D^-(X,\L)$, resp. $D^+(X,\L)$ and resp. $D^b_c(X,\L)$) its full subcategory consisting of objects bounded of finite tor-dimension with constructible cohomologies (resp. of objects bounded above, resp. of objects bounded below and resp. of objects bounded with constructible cohomologies). We denote by $\sK_X$ the complex $\R f^!\L$, where $f:X\ra\spec (k)$ is the structure map and by $\bD_X$ the functor $\R\shom(- ,\sK_X)$ on $\dctf(X,\L)$. For two $k$-schemes $X$ and $Y$, and an \'etale sheaf of $\L$-modules $\sF$ (resp. $\sG$) on $X$ (resp. $Y$), $\sF\bxt\sG$ denotes the sheaf $\pr_1^*\sF\ot\pr^*_2\sG$ on $X\ti_kY$.

\subsection{}
Let $X$ be a scheme and $\sE$ a sheaf of $\sO_X$-modules of finite type. Following (\cite{egaii} 1.7.8), we denote by ${\mathbf V}(\sE)$ the vector bundle $\spec(\mathrm{Sym}_{\sO_X}(\sE))$ over $X$.

\subsection{}\label{locchern}
Let $X$ be a $k$-scheme of equidimension $e$, $Z$ a closed subscheme of $X$, $\sE_1$ and $\sE_0$ locally free $\sO_X$-modules of rank $e$, $f:\sE_1\ra\sE_0$ an $\sO_X$-linear map which is an isomorphism on $X-Z$, and $\sE=[\sE_1\xra{f}\sE_0]$ the complex such that $\sE_0$ is in degree $0$. For $i>0$, we put (\cite{ksann} 3.24)
\begin{equation*}
  c_i(\sE_0-\sE_1)^X_Z=\sum^{\min(e,i-1)}_{j=0}c_j(\sE_1)\cap {c_{i-j}}^X_Z(\sE)
\end{equation*}
as a bivariant class, where $c(\sE_1)$ denotes the Chern class of $\sE_1$ (\cite{fulton} 3.2) and ${c}^X_Z(\sE)$ the localized Chern class of $\sE$ (\cite{fulton} 18.1).

\section{Preliminaries on \'etale cohomology}
\subsection{}
Let
\begin{equation}\label{general bcd}
  \xymatrix{\relax
  X'\ar[r]^{f'}\ar[d]_{g'}&Y'\ar[d]^g\\
  X\ar[r]^f&Y}
\end{equation}
be a commutative diagram of $k$-schemes. We have the base change maps (\cite{sga4} XVII 4.1.5, XVIII 3.1.13.2)
\begin{eqnarray}
 g^*\R f_* &\ra & \R f'_* g'^*, \label{bclow*} \\
 \R f'_!\R g'^! &\ra & \R g^!\R f_!. \label{all!bc}
\end{eqnarray}

Assume that diagram \eqref{general bcd} is Cartesian. There exists a canonical base change isomorphism (the proper base change theorem) (\cite{sga4} XVII 5.2.6)
\begin{equation}\label{pbc}
  g^*\R f_!\isora \R f'_! g'^*.
\end{equation}
There exists a canonical isomorphism of functors (\cite{sga4} XVIII 3.1.12.3)
\begin{equation}\label{low*up!}
\R f'_*\R g'^!\isora\R g^!\R f_*.
\end{equation}
There exists a canonical morphism of functors  (\cite{sga4} XVIII 3.1.14.2)
\begin{equation}
  g'^*\R f^!\ra \R f'^!g^*.
\end{equation}
It is defined as the adjoint of the composed morphisms
\begin{equation*}
  \R f'_!g'^*\R f^!\isora g^*\R f_!\R f^!\ra g^*,
\end{equation*}
where the first arrow is induced by the inverse of the proper base change theorem and the
second arrow is induced by the adjunction map $\R f_!\R f^!\ra \id$.

\subsection{}
Let $f:X\ra Y$ be a morphism of $k$-schemes and $\sF$ (resp. $\sG$) an object of $\dctf(X,\L)$ (resp. $\dctf(Y,\L)$). There exists a canonical isomorphism (the projection formula) (\cite{sga4} XVII 5.2.9)
\begin{equation}\label{projection formula}
\R f_!(f^*\sG\ot^L \sF)\isora \sG\ot^L\R f_!\sF.
\end{equation}

\subsection{}
For a morphism $f:X\ra Y$ of $k$-schemes and two objects $\sF$ and $\sG$ of $\dctf(Y,\L)$, we have a canonical map
\begin{equation}\label{gcup}
f^*\sF\ot^L\R f^!\sG\ra \R f^!(\sF\ot^L\sG),
\end{equation}
defined as follows. By the projection formula \eqref{projection formula}, we have a canonical isomorphism
\begin{equation*}
\R f_!(f^*\sF\ot^L\R f^!\sG)\isora \sF\ot^L\R f_!(\R f^!\sG).
\end{equation*}
Composing with the adjunction map $\R f_!(\R f^!\sG)\ra \sG$, we obtain a map $\R f_!(f^*\sF\ot^L\R f^!\sG)\ra \sF\ot^L\sG$, which gives \eqref{gcup} by adjunction.

If $f$ is a closed immersion, \eqref{gcup} induces a cup product
\begin{equation}\label{cap prod}
 \h^i(X, f^*\sF)\ti \h^j_{X}(Y, \sG)\xra{\cup} \h^{i+j}_{X}(Y, \sF\ot^L\sG).
\end{equation}


\subsection{}\label{loclemma}
Let $g:W\ra X$ and $f:X\ra Y$ be closed immersions of $k$-schemes, and $\sF$ and $\sG$ objects of $D_{\mathrm{ctf}}(Y,\L)$. Then, the following diagram is commutative
\begin{equation}\label{locdiag1}
  \xymatrix{\relax
  f^*\sF\ot^L g_*\R g^!(\R f^!\sG)\ar[d]\ar[r]^-{\sim}&g_*((fg)^*\sF\ot^L \R (fg)^!\sG)\ar[r]^-{\eqref{gcup}}&g_*\R g^!(\R f^!(\sF\ot^L\sG))\ar[d]\\
  f^*\sF\ot^L \R f^!\sG\ar[rr]^-{\eqref{gcup}}& &\R f^!(\sF\ot^L\sG)}
\end{equation}
where the vertical maps are induced by the adjunction map $g_*\R g^!\ra \id$, and the isomorphic map is induced by the projection formula \eqref{projection formula}.
Indeed, it is enough to show that the following diagram is commutative
\begin{equation*}
  \xymatrix{\relax
  f_*(f^*\sF\ot^L g_*\R g^!(\R f^!\sG))\ar[d]\ar[r]^-{\sim}&(fg)_*((fg)^*\sF\ot^L \R (fg)^!\sG)\ar[r]^-{\sim}&\sF\ot^L(fg)_*\R (fg)^!\sG\ar[d]\\
  f_*(f^*\sF\ot^L \R f^!\sG)\ar[r]^-{\sim}&\sF\ot^L f_*(\R f^!\sG) \ar[r]& \sF\ot^L\sG}
\end{equation*}
where the isomorphic maps are the projection formulae and the other maps are induced by adjunction. Since the composition of the upper horizontal maps is the projection formula
\begin{equation*}
  f_*(f^*\sF\ot^L g_*\R g^!(\R f^!\sG))\isora \sF\ot^L(fg)_*\R (fg)^!\sG,
\end{equation*}
we are reduced to show the following diagram is commutative
\begin{equation*}
  \xymatrix{\relax
  f_*(f^*\sF\ot^L (g_*\R g^!(\R f^!\sG)))\ar[d]_{\adj}\ar[r]^-{\sim}&\sF\ot^Lf_*(g_*\R g^! (\R f^!\sG))\ar[d]^{\adj}\\
  f_*(f^*\sF\ot^L \R f^!\sG)\ar[r]^-{\sim}&\sF\ot^L f_*(\R f^!\sG)}
\end{equation*}
which is obvious.

Diagram \eqref{locdiag1} induces a commutative diagram
\begin{equation}\label{locdiag}
\xymatrix{\relax
  \h^i(X,f^*\sF)\ti\h^j_W(Y,\sG)\ar[r]^-{\cup_W}\ar[d]&\h^{i+j}_W(Y,\sF\ot^L\sG)\ar[d]\\
  \h^i(X,f^*\sF)\ti\h^j_X(Y,\sG)\ar[r]^-{\cup}&\h^{i+j}_X(Y,\sF\ot^L\sG)}
\end{equation}
where $\cup_W$ is defined by the upper horizontal arrows of \eqref{locdiag1}.

\subsection{}
Let $f:X\ra Y$ be a morphism of $k$-schemes, $\sF$ an object of $D^-(X,\L)$ and $\sG$ an object of $D^+(Y,\L)$. We have a canonical isomorphism (\cite{sga4} XVIII 3.1.10, \cite{fu} 8.4)
\begin{equation}\label{adj!}
  \R f_*\R\shom(\sF, Rf^!\sG)\isora \R\shom(Rf_!\sF,\sG).
\end{equation}
Taking $\sG=\sK_Y$, we obtain an isomorphism (\ref{notsheaf})
\begin{equation}\label{adj!cor}
  \R f_*(\bD_X(\sF))\isora \bD_Y(\R f_!\sF).
\end{equation}

\subsection{}
For a morphism $f:X\ra Y$ of $k$-schemes and two objects $\sF$ and $\sG$ of $\dctf(Y,\L)$, we recall the definition of the canonical isomorphism (\cite{sga4} XVIII 3.1.12.2, \cite{fu} 8.4.7)
\begin{equation}\label{gpd}
  \R\shom(f^*\sF,\R f^!\sG)\isora \R f^!\R\shom(\sF,\sG).
\end{equation}
By the inverse of the projection formula \eqref{projection formula}, we have a canonical isomorphism
\begin{equation*}
\sF\ot^L\R f_!\R\shom(f^*\sF,Rf^!\sG)\isora \R f_!(f^*\sF\ot^L\R\shom(f^*\sF,\R f^!\sG)).
\end{equation*}
Composing with the canonical morphisms
\begin{equation*}
\R f_!(f^*\sF\ot^L\R\shom(f^*\sF,\R f^!\sG))\ra \R f_!\R f^!\sG\ra\sG,
\end{equation*}
we obtain a morphism
\begin{equation*}
\sF\ot^L\R f_!\R\shom(f^*\sF,\R f^!\sG)\ra \sG.
\end{equation*}
It induces the map \eqref{gpd} by adjunction. Taking $\sG=\sK_Y$, we obtain a canonical isomorphism
\begin{equation}
  \bD_X(f^*\sF)\isora \R f^!(\bD_Y(\sF)).
\end{equation}

\subsection{}
For two $k$-schemes $X$ and $Y$ and objects $\sF$ and $\sG$ of $\dctf(X,\L)$ and $\dctf(Y,\L)$, a canonical isomorphism
\begin{equation}\label{homot}
\R\shom(\pr^*_2\sG,\R\pr^!_1\sF)\isora \sF\bxt^L\bD_Y(\sG)
\end{equation}
is defined in (\cite{sga5} III 3.1.1).

Let $g:X\ra\ol X$ and $h:Y\ra\ol Y$ be open immersions. By \eqref{homot}, \eqref{adj!cor} and the K\"unneth formula, we have a canonical isomorphism on $\ol X\ti_k \ol Y$
\begin{eqnarray}\label{kunforH}
  \R\shom(\pr^*_2h_!\sG,\R\pr^!_1g_!\sF)&\isora& (g_!\sF)\bxt^L\bD_{\ol Y}(h_!\sG)\isora (g_!\sF)\bxt^L(\R h_*(\bD_{Y}(\sG)))\\
  &\isora& (g\ti 1)_!(\R(1\ti h)_*(\sF\bxt^L\bD_Y(\sG)))\nonumber\\
  &\isora& (g\ti 1)_!(\R(1\ti h)_*\R\shom(\pr^*_2\sG,\R\pr^!_1\sF)).\nonumber
\end{eqnarray}

\subsection{}
Let $f:X\ra Y$ be a flat morphism of $k$-schemes with fibers of equidimension $d$ and $\sF$ an object of $\dctf(Y,\L)$. We have a canonical trace map (\cite{sga4} XVIII 2.9)
\begin{equation*}
\mathrm{Tr}_f:\R f_!f^*\sF(d)[2d]\ra\sF.
\end{equation*}
Its adjoint
\begin{equation}\label{classmap}
  t_f:f^*\sF(d)[2d]\ra \R f^!\sF
\end{equation}
is called the class map (\cite{sga4} XVIII 3.2.3). If $f$ is smooth, $t_f$ \eqref{classmap} is an isomorphism (Poincar\'e duality) (\cite{sga4} XVIII 3.2.5, \cite{fu} 8.5.2).

\subsection{}
Let $f:X\ra Y$ be a morphism of smooth $k$-schemes with the same equidimension $d$. For any object $\sF$ of $\dctf(Y, \L)$, we recall the definition of the canonical map (\cite{as} (1.9))
\begin{equation}\label{samedim}
f^*\sF\ra \R f^!\sF.
\end{equation}
The map $f$ is the composition of the graph $\G_f:X\ra X\ti_kY$ of $f$ and the projection $\pr_2:X\ti_kY\ra Y$. Since $\G_f$ is a section of the projection $\pr_1:X\ti_kY\ra X$, there exists a canonical isomorphism $\L\ra \R f^!\L$ defined as the composition
\begin{equation*}
  \L\isora \R\G_f^!\R\pr_1^!\L\isora \R\G_f^!\L(d)[2d]\isora \R\G_f^!\R\pr_2^!\L\isora \R f^!\L,
\end{equation*}
where the second and the third arrows are induced by Poincar\'e duality. Then, the canonical map \eqref{gcup} induces \eqref{samedim}.

\subsection{}\label{cycle map def}
Let $V$ be a $k$-scheme, $Z$ an integral closed subscheme of $V$ of equidimension $d$. The canonical class map $\L(d)[2d]\ra \sK_Z$ induces a morphism
\begin{equation}\label{cycmap}
\h^0(Z,\L)\ra \h^{-2d}_Z(V,\sK_V(-d)).
\end{equation}
The cycle class $[Z]\in \h^{-2d}_Z(V,\sK_V(-d))$ is defined as the image of $1\in \h^0(Z,\L)$ by the map \eqref{cycmap}. We obtain a homomorphism $Z_d(V)\ra \h^{-2d}(V,\sK_V(-d))$, where $Z_d(V)$ denotes the free abelian group generated by integral closed subschemes of equidimension $d$ of $V$. This map factors through the Chow group $\CH_d(V)$, and induces the cycle map
\begin{equation}\label{cycle map cl}
\cl:\CH_d(V)\ra \h^{-2d}(V,\sK_V(-d)).
\end{equation}

If $V$ is a closed subscheme of a smooth $k$-scheme $X$ of dimension $d+c$, by Poincar\'e duality, we have
\begin{equation*}
  \h^{-2d}(V,\sK_V(-d))\isora\h^{-2d}_V(X,\sK_X(-d))\isora\h^{2c}_V(X,\L(c)).
\end{equation*}

Let $Y$ be another smooth $k$-scheme of dimension $e+c$, $f:X\ra Y$ a $k$-morphism and $W$ a closed subscheme of $Y$ such that $f^{-1}(W)$ is a closed subscheme of $V$. By (\cite{ksann} 2.1.2), we have a commutative diagram
\begin{equation}\label{up*gysin!}
  \xymatrix{\relax
  \CH_e(W)\ar[r]^-{\cl}\ar[d]_{f^!}& \h^{2c}_W(Y,\L(c))\ar[d]^{f^*}\\
  \CH_d(V)\ar[r]^-{\cl}&\h^{2c}_V(X,\L(c))}
\end{equation}
where the map $f^!$ denotes the refined Gysin homomorphism (\cite{fulton} 6.6).

\subsection{}\label{acyclic}
Let $X$ be a $k$-scheme, $Z$ a closed subscheme of $X$, $V=X-Z$ the complementary open subscheme of $Z$ in $X$, $U$ an open subscheme of $X$,
$i:Z\ra X$, $j:V\ra X$, $i_U:Z\cap U\ra U$ and $j_U:U\cap V\ra U$ the canonical injections, and $\sF$ an object of $\dctf(X,\L)$. Assume that for any integer $q$, $\sH^ q(\sF)|_U$ is locally constant and constructible. Then, we have a canonical isomorphism (\cite{fu} 6.5.5)
\begin{equation*}
  (\sF|_U)\ot^L\R j_{U*}\L\isora \R j_{U*}j^*_U(\sF|_U).
\end{equation*}
Since $\R i_U^!\R j_{U*}=0$ \eqref{low*up!}, we have $\R i_U^!((\sF|_U)\ot^L\R j_{U*}\L)=0$. In particular, for any integer $q$, the canonical map
\begin{equation*}
  \h^ q_{Z-U}(X,\sF\ot^L\R j_*(\L_V))\ra \h^q_Z(X,\sF\ot^L\R j_*(\L_V))
\end{equation*}
is an isomorphism (\cite{ts2} Lemma 3.1).

\subsection{}\label{delta}
Let $X$ be a $k$-scheme, $Z$ a closed subscheme of $X$, $V=X-Z$ the complementary open subscheme of $Z$ in $X$, $i:Z\ra X$ and $j:V\ra X$ the canonical injections, and $\sF$ an object of $D^-(X,\L)$. Applying the functor $\R i^!(\sF\ot^L-)$ to the distinguished triangle $i_*\R i^!\L_X\ra\L_X\ra \R j_*\L_V\ra$, we obtain a distinguished triangle (\cite{ts2} Lemma 3.5)
 \begin{equation}\label{tridelta}
  i^*\sF\ot^L\R i^!(\L_X)\xra{a} \R i^!\sF\xra{b} \R i^!(\sF\ot^L\R j_*(\L_V))\ra,
 \end{equation}
 where $a$ is the map \eqref{gcup}. We denote the functor $\R i^!(-\ot^L\R j_*(\L_V)):D^-(X,\L)\ra D^-(Z,\L)$ by $\Delta_i(-)$.

\subsection{}\label{killrg}
Let $X$ be a $k$-scheme, $\d:X\ra\xx$ the diagonal map and $S$ a closed subscheme of $X$ such that its complement is dense in $X$. Consider the  composed map
\begin{equation}\label{aslcciso}
  \sK_X\ra\R\d^!\d_*\sK_X\ra\Delta_{\d}(\d_*(\sK_X)),
\end{equation}
where the first arrow is the adjunction map and the second arrow is $b$ in \eqref{tridelta}. If $X$ is an equidimensional smooth $k$-scheme, the following map induced by \eqref{aslcciso} is an isomorphism (\cite{as} 5.2)
\begin{equation}\label{smooth lcc sim}
  \h^0_S(X,\sK_X)\isora\h^0_S(X,\Delta_{\d}(\d_*(\sK_X))).
\end{equation}

\section{Cohomological correspondences}

\begin{definition}[\cite{sga5} III 3.2, \cite{as},1.2.1]\label{cohcor}
Let $X$ and $Y$ be two $k$-schemes. A {\it correspondence} between $X$ and $Y$ is a $k$-scheme $C$ equipped with $k$-morphisms $c_1:C\ra X$ and $c_2:C\ra Y$. Let $\sF$ and $\sG$ be objects of $\dctf(X,\L)$ and $\dctf(Y,\L)$, respectively. A {\it cohomological correspondence} is a morphism $u:c_2^*\sG\ra \R c^!_1\sF$ from $\sG$ to $\sF$ on $C$.
\end{definition}

We switch the factors compared to (\cite{sga5} III 3.2).

\subsection{}\label{procc}
Let $X$ and $Y$ be two $k$-schemes and $(C,c_1:C\ra X, c_2:C\ra Y)$ a correspondence between $X$ and $Y$. We denote by $c$ the map $(c_1,c_2):C\ra X\ti_k Y$ and by $\pr_1:X\ti_k Y\ra X$ and $\pr_2:X\ti_k Y\ra Y$ the canonical projections.
Let $\sF$ and $\sG$ be objects of $\dctf(X,\L)$ and $\dctf(Y,\L)$, respectively. We have a canonical isomorphism \eqref{gpd}
\begin{equation*}
\R\shom(c_2^*\sG,\R c^!_1\sF)\isora \R c^!\R\shom(\pr_2^*\sG,\R\pr_1^!\sF).
\end{equation*}
Taking global sections on $C$, we get a canonical isomorphism
\begin{equation}\label{cohcoreq}
  \Hom(c^*_2\sG,\R c^!_1\sF)\isora\h^0(C,\R c^!\R\shom(\pr_2^*\sG,R\pr_1^!\sF)),
\end{equation}
which shows that cohomological correspondences $u:c_2^*\sG\ra \R c^!_1\sF$ are in one to one correspond with morphisms $\L_C\ra \R c^!\R\shom(\pr_2^*\sG,\R\pr_1^!\sF)$, and hence with morphisms $\R c_!(\L_C)\ra \R\shom(\pr_2^*\sG,\R\pr_1^!\sF)$ by adjunction.

If $c:C\ra X\ti_kY$ is a closed immersion, and $X$ and $Y$ are smooth $k$-schemes of dimension $d$, we have
\begin{equation*}
\h^0(C,\R c^!\R\shom(\pr_2^*\sG,\R\pr_1^!\sF))\isora\h^{2d}_C(X\ti_kY, \R\shom(\pr_2^*\sG,\pr_1^*\sF)(d)).
\end{equation*}
If we further assume that $\sF$ and $\sG$ are sheaves of free $\L$-modules and that $\sG$ is locally constant and constructible. Then the canonical map $c^*\R\shom(\pr^*_2\sG,\pr_1^*\sF)\ra\R\shom(c^*_2\sG,c_1^*\sF)$
is an isomorphism, and we have
\begin{equation*}
 \Hom(c^*_2\sG,c_1^*\sF)\isora\h^0(C,c^*\R\shom(\pr^*_2\sG,\pr_1^*\sF)).
\end{equation*}
Then, the cycle class map $\CH_d(C)\ra \h^{2d}_C(X\ti_k Y,\L(d))$ induces a pairing
\begin{eqnarray*}
 \CH_d(C)\ot\Hom(c^*_2\sG,c^*_1\sF)&\ra & \h^{2d}_C(X\ti_k Y,\L(d))\ot \h^0(C, c^*\R\shom(\pr^*_2\sG,\pr_1^*\sF))\\
 &\xra{\cup}& \h^0_C(X\ti_k Y, \R\shom(\pr_2^*\sG,\R\pr^!_1\sF))=\Hom(c^*_2\sG,\R c^!_1\sF).
\end{eqnarray*}
In this case, for a cycle class $\G\in \CH_d(C)$ and a homomorphism $\gamma:c^*_2\sG\ra c^*_1\sF$, the pair $(\G,\g)$ induces a cohomological correspondence $u(\G,\g)$ from $\sF$ to $\sG$ on $C$.

\subsection{}
We consider a commutative diagram of $k$-schemes
\begin{equation}\label{bc cf}
  \xymatrix{\relax
  X\ar[d]_f&C\ar[r]^{c_2}\ar[l]_{c_1}\ar[d]_h&Y\ar[d]^g\\
  X'&C'\ar[r]^{c'_2}\ar[l]_{c'_1}&Y'}
\end{equation}
and let $\sF$ and $\sG$ be objects of $\dctf(X,\L)$ and $\dctf(Y,\L)$, respectively.
By \eqref{adj!cor}, \eqref{homot} and the K\"unneth formula, we have a canonical isomorphism
\begin{equation}\label{f ti g*}
\R(f\ti g)_*\R\shom(\pr_2^*\sG,\R\pr_1^!\sF)\isora \R\shom(\pr_2^*\R g_!\sG, \R\pr_1^!\R f_*\sF).
\end{equation}
Diagram \eqref{bc cf} gives a commutative diagram
\begin{equation*}
  \xymatrix{\relax
  C\ar[d]_h\ar[r]^-c&X\ti_k Y\ar[d]^{f\ti g}\\
  C'\ar[r]^-{c'}&X'\ti_kY'}
\end{equation*}
We assume that $f$, $g$ and $h$ are proper. A cohomology correspondence $u:c^*_2\sG\ra \R c^!_1\sF$ is identified with a map $u:\L_C\ra \R c^!\R\shom(\pr_2^*\sG,\R\pr^!_1\sF)$ (\ref{procc}). It induces a map
\begin{equation}\label{pullcc}
\L_{C'}\ra \R h_*\R c^!\R\shom(\pr_2^*\sG,\R\pr^!_1\sF).
\end{equation}
The base change map \eqref{all!bc} gives
\begin{equation}\label{probc}
  \R h_*\R c^!=\R h_!\R c^!\ra \R c'^!\R(f\ti g)_!= \R c'^!\R(f\ti g)_*.
\end{equation}
Composing \eqref{f ti g*}, \eqref{pullcc} and \eqref{probc}, we obtain a map
\begin{equation*}
  \L_{C'}\ra \R c'^!R(f\ti g)_*R\shom(\pr_2^*\sG,R\pr^!_1\sF)\isora \R c'^!\R\shom(\pr^*_2\R g_!\sG,\R\pr^!_1\R f_*\sF).
\end{equation*}
By \eqref{cohcoreq}, we obtain a map
\begin{equation*}
 c'^*_2\R g_!\sG=c'^*_2\R g_*\sG\ra \R c'^!_1\R f_*\sF,
\end{equation*}
which is a correspondence form $\R g_*\sG$ to $\R f_*\sF$ on $C'$, that we denote by $h_*(u)$ and call the push-forward of $u$ by $h$. The map $h_*(u)$ is equal to the composition of the maps
\begin{equation*}
  c'^*_2\R g_*\sG\ra \R h_*c_2^*\sG\xra{\R h_*(u)}\R h_*\R c_1^!\sF\ra \R c'^!_1\R f_*\sF,
\end{equation*}
where the left and right arrows are the base change maps.

\subsection{}\label{openpull}
We consider a commutative diagram of $k$-schemes
\begin{equation}\label{opdiag}
  \xymatrix{\relax
  U\ar[d]_{j_U}&C'\ar[r]^{c'_2}\ar[d]^{j_C}\ar[l]_{c'_1}&V\ar[d]^{j_V}\\
  X&C\ar[r]^{c_2}\ar[l]_{c_1}&Y}
\end{equation}
where all the vertical arrows are open immersions. Let $\sF$ and $\sG$ be objects of $\dctf(X,\L)$ and $\dctf(Y,\L)$, respectively, and $u:c_2^*\sG\ra \R c_1^!\sF$  a cohomological correspondence on $C$. Denote by $\sF_U$ and $\sG_V$ the restrictions of $\sF$ and $\sG$ to $U$ and $V$, respectively. We have $\R j_C^!=j_C^*$ and $\R j_U^!=j_U^*$. Hence, the restriction $u'$ of $u$ to $C'$ defines a cohomological correspondence
\begin{equation}
  u':c'^*_2(\sG_V)=j_C^*c_2^*\sG\xra{j_C^*(u)}j^*_C\R c^!_1\sF=\R c'^!_1(\sF_U).
\end{equation}
We denote by $j$ the map $j_U\ti j_V:U\ti_kV\ra X\ti_kY$, by $c$ the map $(c_1,c_2):C\ra X\ti_kY$ and by $c'$ the map $(c'_1,c'_2):C'\ra U\ti_kV$. We have a commutative diagram
\begin{equation}
  \xymatrix{\relax
  C'\ar[r]^-{c'}\ar[d]_{j_C}&U\ti_kV\ar[d]^j\\
  C\ar[r]^-{c}&X\ti_kY}
  \end{equation}
The base change map \eqref{all!bc} gives a canonical morphism
\begin{equation}\label{bc!!}
\R c'_!(\L_{C'})=\R c'_!\R j_C^!(\L_{C}) \ra \R j^!\R c_!(\L_C)=j^*\R c_!(\L_C).
\end{equation}
Put
\begin{eqnarray}
  \sH'&=&\R \shom(\pr_2^*(\sG_V),\R \pr_1^!(\sF_U))\ \ \ {\rm on}\ \ \ U\ti_kV,\\
  \sH&=&\R \shom(\pr_2^*\sG,\R \pr_1^!\sF)\ \ \ {\rm on}\ \ \ X\ti_kY.
\end{eqnarray}
By \eqref{procc}, we identify a cohomological correspondence $u:c_2^*\sG\ra \R c_1^!\sF$ with a map $u: \L_C\ra \R c^!\sH$ and also with the associated map $u: \R c_!(\L_C)\ra\sH$. We identify the restriction $u':c'^*_2(\sG_V)\ra \R c'^!_1(\sF_U)$ of $u$ with a map $u':\L_{C'}\ra \R c'^!(\sH')$ also with the associated map $u':\R c'_!(\L_{C'})\ra\sH'$. Since $\R j_U^!=j_U^*$ and $\R j^!=j^*$, by \eqref{gpd}, we have a canonical isomorphism
\begin{equation}\label{jHH'iso}
  j^*\sH\isora\sH'.
\end{equation}

\begin{lemma}[\cite{as} Lemma 1.2.2]
We take the notation and assumptions of \eqref{openpull}. Then,
\begin{itemize}
  \item[1.]
  The map $u':\L_{C'}\ra \R c'^!(\sH')$ coincides with the restriction of $u:\L_C\ra \R c^!\sH$ to $C'$ by the composed isomorphism $j^*_C \R c^!\sH=\R j_C^!\R c^!\sH\ra \R c'^!\R j^!\sH=\R c'^!j^*\sH=\R c'^!(\sH')$.
  \item[2.]
  The following diagram is commutative
  \begin{equation}
    \xymatrix{\relax
    j^*\R c_!(\L_C)\ar[r]^{j^*u}&j^*\sH\ar[d]^{\eqref{jHH'iso}}\\
    \R c'_!(\L_{C'})\ar[u]^{\eqref{bc!!}}\ar[r]^{u'}&\sH'}
    \end{equation}

\end{itemize}
\end{lemma}

\begin{lemma}[\cite{as} Lemma 1.2.3]\label{0extlemma}
Consider diagram \eqref{opdiag} again and assume moreover that its right square is Cartesian. Let $\sF'$ and $\sG'$ be objects of $\dctf(U,\L)$ and $\dctf(V,\L)$, and $u':c'^*_2(\sG')\ra \R c'^!_1(\sF')$ a cohomological correspondence on $C'$. Then, there exist a unique cohomological correspondence $u:c_2^*j_{V!}(\sG')\ra \R c_1^!j_{U!}(\sF')$ on $C$ such that its restriction to $C'$ is $u'$.
\end{lemma}
We call $u$ in (\ref{0extlemma}) the {\it extension by zero} of $u'$ and denote it by $j_{C!}u'$.

\subsection{}
Let $X$, $X'$, $Y$, $Y'$ be smooth equidimensional $k$-schemes such that $\dim X=\dim X'$ and $\dim Y=\dim Y'$, $f:X'\ra X$ and $g:Y'\ra Y$ morphisms of $k$-schemes, $(C,c_1:C\ra X, c_2:C\ra Y)$ a correspondence between $X$ and $Y$, and $\sF$ and $\sG$ objects of $\dctf(X,\L)$ and $\dctf(X,\L)$, respectively. We denote by $c$ the map $(c_1,c_2):C\ra X\ti_k Y$. By \eqref{gpd} and \eqref{samedim}, we have a map
\begin{eqnarray}
(f\ti g)^*\R \shom(\pr_2^*\sG,\R \pr_1^!\sF)&\ra&\R (f\ti g)^!\R \shom(\pr_2^*\sG,\R \pr_1^!\sF) \label{f!f*}\\
&\ra&\R \shom(\pr^*_2g^*\sG,\R \pr^!_1\R f^!\sF).\nonumber
\end{eqnarray}
Let $u:c_2^*\sG\ra \R c_1^!\sF$ be a cohomological correspondence on $C$, that we identify with a map $u:\R c_!\L_{C}\ra \R \shom(\pr_2^*\sG,\R \pr_1^!\sF)$. We define a correspondence $c'=(c'_1,c'_2):C'\ra X'\ti_k Y'$ by the Cartesian diagram
\begin{equation*}
  \xymatrix{\relax
  C'\ar@{}[rd]|{\Box}\ar[r]^-{c'}\ar[d]_h&X'\ti_k Y'\ar[d]^{f\ti g}\\
  C\ar[r]^-{c}&X\ti_kY}
\end{equation*}
By the proper base change theorem, the base change map $(f\ti g)^*\R c_!\L_C\ra \R c'_!\L_{C'}$ is an isomorphism. The composed map
\begin{eqnarray*}
  \R c'_!(\L_{C'})\xra{\sim} (f\ti g)^*\R c_!(\L_C)&\ra &(f\ti g)^*\R \shom(\pr_2^*\sG,\R \pr_1^!\sF)\\
  &\ra&\R \shom(\pr_2^*g^*\sG,\R \pr_1^!\R f^!\sF),
\end{eqnarray*}
where the first arrow is the inverse of the base change isomorphism, corresponds to a cohomological correspondence
\begin{equation*}
(f\ti g)^*(u):c'^*_2g^*\sG\ra \R c'^!_1\R f^!\sF,
\end{equation*}
 called the pull-back of $u$ by $f\ti g$.

\subsection{}
Let $X$ be a $k$-scheme, $\sF$ an object of $\dctf(X,\L)$. We denote by $\d:X\ra\xx$ the diagonal map and put $\sH=\R \shom(\pr^*_2\sF,\R \pr_1^!\sF)$. The canonical isomorphism
$\sH\isora\sF\bxt^L\bD_X(\sF)$ \eqref{homot} induces an isomorphism
\begin{equation*}
  \d^*\sH\isora\sF\ot^L\bD_X(\sF).
\end{equation*}
Composed with the evaluation map $\sF\ot^L\bD_X(\sF)\ra\sK_X$, we get a map
\begin{equation}\label{ev cf}
  \ev:\d^*\sH\ra\sK_X,
\end{equation}
that we also call the evaluation map.

Let $C$ be a closed subscheme of $\xx$ and $u$ a cohomological correspondence of $\sF$ on $C$. We denote by $c:C\ra\xx$ the canonical injection. By (\ref{procc}), $u$ corresponds to a section
\begin{equation*}
  u\in \h^0(C,\R c^!\sH)=\h^0_C(\xx,\sH).
\end{equation*}
We call the image of $u$ by the following composed maps
\begin{equation*}
\h^0_C(\xx,\sH)\xra{\d^*}\h^0_{C\cap X}(X,\d^*\sH)\xra{\ev}\h^0_{C\cap X}(X,\sK_X)
\end{equation*}
the {\it characteristic class of the cohomological correspondence} $u$, and denote it by $C(\sF,C,u)\in \h^0_{C\cap X}(X,\sK_X)$ (\cite{sga5} III, \cite{as} 2.1.8). If $C=\d(X)$, and $u:\sF\ra\sF$ is an endomorphism (resp. the identity of $\sF$), we abbreviate the notation of the characteristic class of $u$ by $C(\sF,u)\in \h^0(X,\sK_X)$ (resp. $C(\sF)\in\h^0(X,\sK_X)$, and call it the {\it characteristic class of} $\sF$).

\subsection{}
Let $X$ be a $k$-scheme, $U$ an open subscheme of $X$, and $\sF$ an object of $\dctf(U,\L)$. We denote by $j:U\ra X$ the canonical open immersion, and by $\d:X\ra\xx$ and $\d_U:U\ra U\ti_kU$ the diagonal maps. Put
\begin{eqnarray}
  \sH&=&\R \shom(\pr^*_2\sF,\R \pr_1^!\sF)\ \ \ {\rm on}\ \ \ U\ti_kU,\nonumber\\
  \ol\sH&=&\R \shom(\pr^*_2j_!\sF,\R \pr_1^!j_!\sF)\ \ \ {\rm on}\ \ \ X\ti_kX.\nonumber
\end{eqnarray}
By \eqref{homot} and the projection formula for $j_!$ \eqref{projection formula}, we have
\begin{equation*}
  \d^*(\ol\sH)\iso(j_!\sF)\ot^L\bD(j_!\sF)\iso j_!(\sF\ot^L\bD(\sF))\iso j_!(\d_U^*\sH).
\end{equation*}
Then, the evaluation map $\ev:\d_U^*\sH\ra\sK_U$ \eqref{ev cf} induces a map
\begin{equation*}
\ev':\d^*(\ol\sH)\ra j_!(\sK_U).
\end{equation*}

Let $C$ be a closed subscheme of $\uu$ and $u$ a cohomological correspondence of $\sF$ on $C$. We denote by $\ol C$ the closure of $C$ in $\xx$. We have a commutative diagram
\begin{equation*}
  \xymatrix{\relax
  C\ar[r]^-{c}\ar[d]_{j_C}&\uu\ar[d]^{j\ti j}\\
  \ol C\ar[r]^-{\ol c}& \xx}
\end{equation*}
where $j$, $c$ and $\ol c$ are the canonical injections. Assume $C=(X\ti_kU)\cap\ol C$. The extension by zero $j_{C!}(u)$ of $u$ \eqref{0extlemma} corresponds, by \eqref{procc}, to a section
\begin{equation*}
 j_{C!}(u)\in \h^ 0(\ol C, \R \ol c^!(\ol\sH))=\h^ 0_{\ol C}(\xx,\ol\sH).
\end{equation*}
We denote by $C_!(j_!\sF,\ol C,j_{C!}(u))$ the image of $j_{C!}(u)$ by the composed map
\begin{equation*}
  \h^ 0_{\ol C}(\xx,\ol\sH)\xra{\d^*}\h^ 0_{\ol C\cap X}(X,\d^*(\ol\sH))\xra{\ev'}\h^ 0_{\ol C\cap X}(X,j_!(\sK_U)).
\end{equation*}
By (\cite{as} 2.1.7), the characteristic class $C(j_{U!}\sF,\ol C,j_{C!}(u))\in \h^0_{\ol C\cap X}(X,\sK_X)$ is the canonical image of
$C_!(j_!\sF,\ol C,j_{C!}(u))$.

\subsection{}\label{lcc as}
Let $X$ be an equidimensional smooth $k$-scheme, $S$ a closed subscheme of $X$, $U=X-S$ the complementary open subscheme of $S$ in $X$ that we assume to be dense in $X$, $j:U\ra X$ the canonical injection, $\d:X\ra\xx$ the diagonal map, and $\sF$ an object of $\dctf(X,\L)$ such that for any integer $q$, $\sH^q(\sF)|_U$ is locally constant. Put $\sH=\R \shom(\pr_2^*\sF,\R \pr_1^!\sF)$ on $\xx$, we have $\R \d^!\sH\isora \R \shom(\sF,\sF)$ \eqref{gpd}. Hence $\id_{\sF}\in\End(\sF)$ corresponds to a map $\L_X\ra \R \d^!\sH$, and by adjunction to a map $\d_*\L_X\ra\sH$. Let $\ev':\sH\ra\d_*\sK_X$ be the adjoint of the evaluation maps \eqref{ev cf}. Applying the functor $\D_{\d}$ \eqref{killrg} to the composition of the two morphisms above, we obtain a map
\begin{equation}\label{keylcc}
\D_{\d}(\d_*(\L_X))\ra \D_{\d}(\sH)\xra{\D_{\d}(\ev')} \D_{\d}(\d_*(\sK_X)).
\end{equation}
Since, for each integer $q$, $\sH^q(\sH)|_{\uu}$ is locally constant and constructible, by (\ref{acyclic}), we have
\begin{equation}\label{isolcckey}
  \h^0_S(X,\D_{\d}(\sH))\isora \h^0(X,\D_{\d}(\sH)).
\end{equation}
Hence, the canonical map $\L_X=\R \d^!\d_*\L_X\ra\D_{\d}(\d_*\L_X)$, \eqref{keylcc}, \eqref{isolcckey} and the inverse of \eqref{smooth lcc sim} define a map
\begin{equation*}
\h^0(X,\L_X)\ra\h^0(X,\D_{\d}(\sH))\isora\h^0_S(X,\D_{\d}(\sH))\ra  \h^0_S(X,\D_{\d}(\d_*\sK_X))\isora\h^0_S(X,\sK_X).
\end{equation*}
We denote the image of $1\in \h^0(X,\L_X)$ in $\h^0_S(X,\sK_X)$ by $C^0_S(\sF)$ and call it the {\it localized characteristic class of $\sF$} (\cite{as} 5.2). In \cite{ts2}, the author gave another definition of the localized characteristic class and proved that the two definitions are equivalent.

\section{Ramification of $\ell$-adic sheaves}

\subsection{}\label{note}
Let $K$ be a complete discrete valuation field, $\sO_K$ the integer ring, $F$ the residue field of $\sO_K$, $\ol K$ a separable closure of $K$, and $G_K$ the Galois group of $\ol K$ over $K$. Abbes and Saito defined two decreasing filtrations $\gk^r$ and $\gkl^r$ ($r\in\bQ_{>0}$) of $G_K$ by closed normal subgroups called the ramification filtration and the logarithmic ramification filtration, respectively (\cite{as i}, 3.1, 3.2).
We denote by $\gkl^0$ the inertia subgroup of $\gk$. For any $r\in \bQ_{\geqslant 0}$, we put
\begin{equation*}
\gkl^{r+}=\ol{\bigcup_{s\in\bQ_{>r}}\gkl^s} \quad\mathrm{and}\quad \grl=\gkl^r\big/\gkl^{r+}.
\end{equation*}
By (\cite{as i} 3.15), $P=\gkl^{0+}$ is the wild inertia subgroup of $G_{K}$, i.e. the $p$-Sylow subgroup of $G^0_{K,\log}$. For every rational number $r>0$, the group $\grl$ is abelian and is contained in the center of $P/\gkl^{r+}$ (\cite{as ii} Theorem 1).

\subsection{}\label{bounded cf}
Let $L$ be a finite separable extension of $K$. For a rational number $r\geqslant 0$, we say that the logarithmic ramification of $L/K$ is bounded by $r$ (resp. by $r+$) if $\gkl^r$ (resp. $\gkl^{r+}$) acts trivially on $\Hom_K(L, \ol K)$ via its action on $\ol K$. The {\it logarithmic conductor} $c$ of $L/K$ is defined as the infimum of rational numbers $r>0$ such that the logarithmic ramification of $L/K$ is bounded by $r$. Then $c$ is a rational number and the logarithmic ramification of $L/K$ is bounded by $c+$ (\cite{as i} 9.5). If $c>0$, the logarithmic ramification of $L/K$ is not bounded by $c$.

\begin{lemma}[\cite{katz} 1.1]
Let $M$ be a $\L$-module on which $P=\gkl^{0+}$ acts $\L$-linearly through a finite discrete quotient, say by $\r: P\ra \aut_{\L}(M)$. Then,
\begin{itemize}
\item[(i)]
The module $M$ has a unique direct sum decomposition
\begin{equation}\label{slopedecom cf}
M=\bp_{r\in\bQ_{\geqslant0}} M^{(r)}
\end{equation}
into $P$-stable submodules $M^{(r)}$, such that $M^{(0)}=M^P$ and for every $r>0$,
\begin{equation*}
(M^{(r)})^{\gkl^r}=0\quad \mathrm{and}\quad (M^{(r)})^{\gkl^{r+}}=M^{(r)}.
\end{equation*}

\item[(ii)] If $r>0$, then $M^{(r)}=0$ for all but the finitely many values of $r$ for which $\r(\gkl^{r+})\neq\r(\gkl^r)$.

\item[(iii)] For any $r\geqslant 0$, the functor $M\ma M^{(r)}$ is exact.

\item[(iv)] For $M$, $N$ as above, we have $\Hom_{P-\mathrm{Mod}}(M^{(r)}, N^{(r')})=0$ if $r\neq r'$.

\end{itemize}
\end{lemma}

\subsection{}\label{slope center decomp cf}
The decomposition \eqref{slopedecom cf} is called the {\em slope decomposition} of $M$. The values $r\geqslant0$ for which $M^{(r)}\neq 0$ are called the {\em slopes} of $M$. We say that $M$ is {\it isoclinic} if it has only one slope. If $M$ is isoclinic of slope $r>0$, we have a canonical central character decomposition
\begin{equation*}
  M=\oplus_{\chi}M_{\chi},
\end{equation*}
where the sum runs over finite characters $\chi:\mathrm{Gr}^r_{\log}G_K\ra\L_{\chi}^{\ti}$ such that $\L_{\chi}$ is a finite \'etale $\L$-algebra (\cite{rc} 6.7).

\subsection{}\label{rsw intro cf}
We assume that $K$ has characteristic $p$ and that $F$ is of finite type over $k$.
Let $\O^1_{F}(\log)$ be the $F$-vector space
\begin{equation*}
\O^1_F(\log)=(\O^1_{F/k}\oplus(F\ot_{\mathbb{Z}} K^{\ti}))/(\mathrm{d} \bar a-\bar a\ot a\,;\,a\in\sO_K^{\ti}),
\end{equation*}
where $\ol a$ denotes the residue class of an element $a\in F$. We denote by $\sO_{\ol K}$ the integral closure of $\sO_K$ in $\ol K$, $\ol F$ the residue field of $\sO_{\ol K}$ and by $v$ the valuation of $\ol K$ normalized by $v(K^{\ti})=\mathbb Z$. For a rational number $r$, we put $\mathfrak{m}^r_{\ol K}$ (resp. $\mathfrak{m}^{r+}_{\ol K}$) the set of elements of $\ol K$ such that $v(x)\geqslant r$ (resp. $v(x)>r$). For any rational number $r>0$, $\mathrm{Gr}^r_{\log}G_K$ is a $\mathbb{F}_p$-vector space, and there exists a canonical injective homomorphism, called the {\em refined Swan conductor} (\cite{saito cc} 1.24),
\begin{equation}\label{rsw llcc}
\mathrm{rsw}:\Hom_{\mathbb{F}_p}(\mathrm{Gr}^r_{\log}G_K,\mathbb{F}_p)\ra \O^1_F(\log)\ot_F\mathfrak{m}^{-r}_{\ol K}/\mathfrak{m}^{-r+}_{\ol K}.
\end{equation}

\subsection{}\label{dilatation}
Let $X$ be a smooth $k$-scheme, $D$ a divisor with simple normal crossing on $X$, $\{D_i\}_{i\in I}$ the irreducible components of $D$. A {\it rational divisor on} $X$ {\it with support in} $D$ is an element $R=\sum_{i\in I}r_iD_i$ of the $\mathbb{Q}$-vector space generated by $\{D_i\}_{i\in I}$. We say that $R$ is {\it effective} if $r_i\ddy 0$ for all $i$. We call {\it generic points} of $R$ the generic points of the $D_i$'s such that $r_i\neq 0$. We denote by $\lfloor nR\rfloor$ the divisor $\sum_{i\in I}\lfloor nr_i\rfloor D_i$ on $X$, where $\lfloor nr_i\rfloor$ is the integral part of $nr_i$. For two rational divisors $R$ and $R'$ with support in $D$, we say that $R'$ is {\it bigger} than $R$ and use the notation $R'\ddy R$ if $R'-R$ is effective.

Let $u:P\ra X$ a smooth separated morphism of finite type, $s:X\ra P$ a section of $u$ and $R$ an effective rational divisor with support on $D$. Put $U=X-D$ and denote by $j:U\ra X$ and $j_P:P_U=u^{-1}(U)\ra P$ the canonical injections and by $\sI_X$ the ideal sheaf of $\sO_P$ associated to $s$. We call {\it dilatation of} $P$ {\it along} $s$ {\it of thickening} $R$ and denote by $P^{(R)}$ the affine scheme over $P$ defined by the quasi-coherent sub-$\sO_P$-algebra of $j_{P*}(\sO_{P_U})$
\begin{equation}\label{dilalg}
  \sum_{n\ddy 0}u^*(\sO_X(\lfloor nR\rfloor))\cdot\sI_X^n.
\end{equation}
The image of the algebra \eqref{dilalg} by the surjective homomorphism $j_{P*}(\sO_{P_U})\ra s_* j_*(\sO_U)$ is canonically isomorphic to $s_*(\sO_X)$. Hence we have a canonical section
\begin{equation*}
  s^{(R)}:X\ra P^{(R)}
\end{equation*}
lifting $s$ (\cite{rc} 5.26).

\subsection{}\label{xstarx}
In the following of this section, let $X$ be a smooth $k$-scheme, $D$ a divisor with simple normal crossing on $X$, $\{D_i\}_{i\in I}$ the irreducible components of $D$, and $j:U=X-D\ra X$ the canonical injection. We denote by $\xxlb_i$ the blow-up of $\xx$ along $D_i\ti_i D_i$, by $(X\rtimes_kX)_i$ the complement of the proper transform of $D\ti_kX$
in $\xxlb_i$ and by $(\xxlp)_i$ the complement of the proper transform of $D_i\ti_kX$ and $X\ti_k D_i$ in $\xxlb_i$. We denote by $\xxlb$ the fiber product of $\{\xxlb_i\}_{i\in I}$ over $\xx$, which is also the blow-up of $\xx$ along $\{D_i\ti_kD_i\}_{i\in I}$ (\cite{saito cc} \S 2.3). We denote by $X\rtimes_k X$ the fiber product of $\{(X\rtimes_kX)_i\}_{i\in I}$ that we call {\it the left-framed self-product of $X$ along $D$}. We denote by $\xxlp$ the fiber product of $\{(\xxlp)_i\}_{i\in I}$ over $\xx$, which is the open subscheme of $\xxlb$ obtained by removing the strict transforms of $D\ti_kX$ and $X\ti_kD$ in $\xxlb$, that we call {\it the framed self-product of} $X$ {\it along} $D$ (\cite{rc} 5.22).

By the universality of the blow-up, the diagonal map $\d:X\ra\xx$ induces closed immersions that we denote by
\begin{equation*}
\d':X\ra\xxlb\ \ \ {\rm and}\ \ \ \wt\d:X\ra\xxlp.
\end{equation*}
We consider $\xxlp$ as an $X$-scheme by the second projection. This projection is smooth
(\cite{saito cc} \S 2.3).

We denote by $D'_i$ the pull-back of $\d(D_i)$ by the canonical projection $\xxlb_i\ra\xx$ and by $D'$ the pull-back of $\d(D)$ by the canonical projection $\xxlb\ra\xx$. By definition, $D'_i\ra D_i$ is a $\mathbb{P}^1$-bundle. For a subset $J$ of $I$, we put $D_J=\bigcap_{i\in J}D_i$ and denote by $n_J$ the cardinality of $J$. Since $\xxlb$ is the fiber product of $\{\xxlb_i\}_{i\in I}$ over $\xx$, $D'$ is the union of $(\mathbb{P}^1)^{n_J}$-bundles over $D_J$ (\cite{ts2} 3.12).

We denote by $\wt D_i$ the pull-back of $\d(D_i)$ by the canonical projection $(\xxlp)_i\ra\xx$ and by $\wt D$ the pull-back of $\d(D)$ by the canonical projection $\xxlp\ra\xx$. By definition, $\wt D_i\ra D_i$ is a $\mathbb{G}_m$-bundle. Since $\xxlp$ is the fiber product of $\{(\xxlp)_i\}_{i\in I}$ over $\xx$, $\wt D$ is the union of $(\mathbb{G}_m)^{n_J}$-bundles over $D_J$ (\cite{ts1} 2.1).
\subsection{}\label{xxr}
For any effective rational divisor $R$ on $X$ with support on $D$, we denote by $\xxr$ the dilatation of $\xxlp$ along $\wt\d$ of thickening $R$ (\ref{dilatation} and \ref{xstarx}). If we consider $\xxlp$ as an $X$-scheme by the first projection, then the dilatation of $\xxlp$ along $\wt\d$ of thickening $R$ is equal to $\xxr$ (\cite{rc} 5.31). There is a canonical morphism
\begin{equation*}
  \d^{(R)}:X\ra\xxr
\end{equation*}
lifting $\wt\d$, and a canonical open immersion
\begin{equation*}
  j^{(R)}:\uu\ra\xxr.
\end{equation*}
Moreover, the following diagram
\begin{equation}\label{xxrcar}
  \xymatrix{\relax
  U\ar@{}[rd]|{\Box}\ar[r]^-{\d_U}\ar[d]_j&\uu\ar[d]^{j^{(R)}}\\
  X\ar[r]^-{\d^{(R)}}&\xxr}
\end{equation}
is Cartesian.

If $R$ has integral coefficients, then the canonical projection $(\xx)^{(R)}\ra X$ is smooth (\cite{rc} 4.6) and we have a canonical $R$-isomorphism (\cite{rc} 4.6.1)
\begin{equation}\label{vectorbundle xxr}
\xxr\ti_X R\isora \mathbf{V}(\Omega^1_{X/k}(\log D)\ot_{\sO_X}\sO_X(R))\ti_XR.
\end{equation}

\subsection{}\label{sheaf bound}
Let $\sF$ be a locally constant constructible sheaf of $\L$-modules on $U$, $R$ an effective rational divisor on $X$ with support on $D$, and $\ol x$ a geometric point of $X$. Put $\sH=\shom(\pr_2^*\sF,\pr_1^*\sF)$ on $\uu$. Then the base change map
\begin{equation}\label{alpha cf}
  \a:\d^{(R)*}j^{(R)}_*(\sH)\ra j_*\d^*_U(\sH)=j_*(\send(\sF))
\end{equation}
relatively to the Cartesian diagram \eqref{xxrcar} is injective (\cite{rc} 8.2). We say that the {\it ramification of} $\sF$ at $\ol x$ {\it is bounded by} $R+$ (\cite{rc} 8.3) if $\sF$ satisfies the following equivalent conditions (\cite{rc} 8.2):
\begin{itemize}
  \item[(i)]
  The stalk $\a_{\ol x}$ of the morphism $\a$ \eqref{alpha cf} at $\ol x$ is an isomorphism.
  \item[(ii)]
  The image of $\id_{\sF}$ in $j_*(\send(\sF))_{\ol x}$ is contained in the image of $\a_{\ol x}$.
\end{itemize}
We say that the {\it ramification of} $\sF$ {\it along} $D$ {\it is bounded by} $R+$ (\cite{rc} 8.3) if the ramification of $\sF$ at $\ol x$ is bounded by $R+$ for every geometric point $\ol x\in X$.

\subsection{}\label{bounded conductor}
Let $\sF$ be a locally constant constructible sheaf of $\L$-modules on $U$, $R$ an effective rational divisor on $X$ with support in $D$, $\xi$ a generic point of $D$, $\ol\xi$ a geometric point of $X$ above $\xi$, $X_{(\ol\xi)}$ the corresponding strictly local scheme, $\eta$ its generic point and $r$ the multiplicity of $R$ at $\xi$.
Then the following conditions are equivalent (\cite{rc} 8.8):
 \begin{itemize}
   \item[(i)]
   The ramification of $\sF$ at $\ol\xi$ is bounded by $R+$.
   \item[(ii)]
   The sheaf $\sF|_{\eta}$ is trivialized by a finite \'etale connected covering $\eta'$ of $\eta$ such that the logarithmic ramification of $\eta'/\eta$ is bounded by $r+$ (\ref{bounded cf}).
 \end{itemize}

The {\it conductor} of $\sF$ at $\xi$ is defined to be the minimum of the set of rational numbers $r\ddy 0$ such that $\sF|_{\eta}$ is trivialized by a finite \'etale connected covering $\eta'$ of $\eta$ and that the logarithmic ramification of $\eta'/\eta$ is bounded by $r+$ (\ref{bounded cf}).
The {\it conductor of} $\sF$ {\it relatively to} $X$ is defined to be the effective rational divisor on $X$ with support in $D$ whose multiplicity at any generic point $\xi$ of $D$ is the conductor of $\sF$ at $\xi$ (\cite{rc} 8.10).

\begin{definition}[\cite{illusie} 2.6]\label{def tame}
Let $Y$ be a $k$-scheme, $Z$ a closed subscheme of $Y$, $V=Y-Z$ the complementary open subscheme of $Z$ in $Y$ that is connected and smooth over $\spec(k)$ and $\sF$ a locally constant and constructible sheaf of $\L$-modules on $V$. For any geometric point $\ol y$ of $Y$, $Y_{(\ol y)}$ denotes the strict localization of $Y$ at $\ol y$.
Then $\sF$ is tamely ramified along $Z$ if the following equivalent conditions are satisfied:
\begin{itemize}
\item[(i)]
For each geometric point $\ol y$ of $Y$ and each geometric point $\ol x\in Y_{(\ol y)}\ti_XV$, the $p$-sylow sub-groups of the \'etale fundamental group $\pi_1(Y_{(\ol y)}\ti_YV,\ol x)$ act trivially on $\sF_{\ol x}$.
  \item[(ii)]
  For each geometric point $\ol y$ of $Y$, there exists an \'etale neighborhood $W$ of $\ol y$ and a Galois \'etale covering $T$ of $W\ti_YV$ of order prime to $p$, such that the pull-back of $\sF$ on $T$ is a constant sheaf.
\end{itemize}
Moreover, if $Y$ is smooth over $\spec(k)$ and $Z$ is a divisor with simple normal crossing on $Y$, $\sF$ is tamely ramified along $Z$ if and only if
\begin{itemize}
  \item[(iii)]
  For any geometric point $\ol\xi$ of $Z$ localized at a generic point of $Z$, the pull-back of $\sF$ on the generic point of the trait $X_{(\ol\xi)}$ is tamely ramified in the usual sense.
\end{itemize}
\end{definition}

\begin{lemma}[\cite{saito cc} 2.21]\label{tame ram}
Let $\sF$ be a locally constant and constructible sheaf of $\L$-modules on $U$. Then the following conditions are equivalent:
\begin{itemize}
  \item[(i)]
  $\sF$ is tamely ramified along $D$.
  \item[(ii)]
  The conductor of $\sF$ vanishes.
  \item[(iii)]
  The ramification of $\sF$ along $D$ is bounded by $0+$.
\end{itemize}
\end{lemma}

\subsection{}\label{isoclinic}
Let $\sF$ be a locally constant and constructible sheaf of $\L$-modules on $U$. Let $\xi$ be a generic point of $D$, $X_{(\xi)}$ the henselization of $X$ at $\xi$, $\eta_{\xi}$ the generic point of $X_{(\xi)}$, $\ol\eta_{\xi}$ a geometric generic point of $X_{(\xi)}$ and $G_{\xi}$ the Galois group of $\ol\eta_{\xi}$ over $\eta_{\xi}$. We say that $\sF$ is {\it isoclinic at} $\xi$ if the representation $\sF_{\ol \eta_{\xi}}$ of $G_{\xi}$ is isoclinic \eqref{slope center decomp cf}. We say that $\sF$ is {\it isoclinic along} $D$ if it is isoclinic at all generic points of $D$ (\cite{rc} 8.22).

\section{Clean $\ell$-adic sheaves and characteristic cycles}

\begin{definition}[\cite{rc} 3.1]\label{additive}
  Let $X$ be a $k$-scheme, $\pi:E\ra X$ a vector bundle, and $\sF$ a constructible sheaf of $\L$-modules on $E$. We say that $\sF$ is additive if for every geometric point $\ol x$ of $X$ and for every $e\in E(\ol x)$, denoting by $\tau_e$ the translation by $e$ on $E_{\ol x}=E\ti_Y\ol x$, $\tau_e^*(\sF|_{E_{\ol x}})$ is isomorphic to $\sF|_{E_{\ol x}}$.
\end{definition}

\subsection{}
Let $\sL_{\psi}$ be the Artin-Schreier sheaf of $\L$-modules of rank $1$ over the additive group scheme $\mathbb{A}^1_{\bF_p}$ over $\bF_p$, associated to the character $\psi$ (\ref{notsheaf}) (\cite{lautf} 1.1.3). If $\mu:\A^1_{\bF_p}\ti_{\bF_p}\A^1_{\bF_p}\ra \A^1_{\bF_p}$ denotes the addition, we have an isomorphism
\begin{equation*}
  \mu^*\sL_{\psi}\isora\pr_1^*\sL_{\psi}\ot\pr_2^*\sL_{\psi}.
\end{equation*}
Hence, $\sL_{\psi}$ is additive (\ref{additive}). If $f:X\ra \A^1_{\bF_p}$ is a morphism of schemes, we put $\sL_{\psi}(f)=f^*\sL_{\psi}$.

\subsection{}
Let $X$ be a $k$-scheme, $\pi:E\ra X$ a vector bundle of constant rank $d$ and $\check\pi:\check E\ra X$ its dual bundle. We denote by $\langle\,\, ,\,\,\rangle:E\ti_X\check E\ra\A^1_{\bF_p}$ the canonical pairing, by $\pr_1:E\ti_X\check E\ra E$ and $\pr_2:E\ti_X\check E\ra \check E$ the canonical projections and by
\begin{equation*}
  \fF_{\psi}:D^b_c(E,\L)\ra D^b_c(\check E,\L)
\end{equation*}
the Fourier-Deligne transform defined by (\cite{lautf} 1.2.1.1)
\begin{equation*}
  \fF_{\psi}(K)=\R\pr_{2!}(\pr_1^*K\ot\sL_{\psi}(\langle\,\,,\,\,\rangle)).
\end{equation*}

Let $\pi^{\flat}:E^{\flat}\ra X$ be the bidual vector bundle of $\pi:E\ra X$, $a:E\ra E^{\flat}$ the anti-canonical isomorphism defined by $a(x)=-\langle x,\,\,\rangle$, and $\fF^{\vee}_{\psi}$ the Fourier-Deligne transform for $\check\pi:\check E\ra X$. For every object $K$ of $D^b_c(E,\L)$, we have a canonical isomorphism (\cite{lautf} 1.2.2.1)
\begin{equation}\label{ftdual}
  \fF^{\vee}_{\psi}\circ\fF_{\psi}(K)\isora a^*(K)(-d)[-2d].
\end{equation}

Let $\pi':E'\ra X$ be a vector bundle of constant rank $d'$, $\fF'_{\psi}$ its Fourier-Deligne transform, $f:E\ra E'$ a morphism of vector bundles, and $\check f:\check E'\ra \check E$ its dual. For every object $K'$ of $D^b_c(E',\L)$, we have canonical isomorphisms (\cite{rc} 3.4.6, 3.4.7)
\begin{eqnarray}
  \R \check f_!\circ\fF'_{\psi}(K')(d')[2d']&\isora& \fF_{\psi}\circ f^*(K')(d)[2d],\label{ftprop1}\\
  \R \check f_*\circ\fF'_{\psi}(K')&\isora&\fF_{\psi}\circ \R f^!(K').\label{ftprop2}
\end{eqnarray}

\subsection{}
Let $X$ be a $k$-scheme and $K$ an object of $D^b_c(X,\L)$. The {\it support} of $K$ is the subset of points of $X$ where the stalks of the cohomology sheaves of $K$ are not all zero. It is constructible in $X$.

\begin{proposition}[\cite{rc} 3.6]
 Let $X$ be a $k$-scheme, $\pi:E\ra X$ a vector bundle of constant rank, $\check\pi:\check E\ra X$ its dual bundle, $\sF$ a constructible sheaf of $\L$-modules on $E$ and $S\subset \check E$ the support of $\fF_{\psi}(\sF)$. Then, $\sF$ is additive if and only if for every $x\in X$, the set $S\cap \check E_x$ is finite.
\end{proposition}

\begin{definition}[\cite{rc} 3.8]\label{fourier ds}
Let $X$ be a $k$-scheme, $\pi:E\ra X$ a vector bundle of constant rank, $\check\pi:\check E\ra X$ its dual bundle, and $\sF$ an additive constructible sheaf of $\L$-modules on $E$. We call the {\it Fourier dual support} of $\sF$ the support of $\fF_{\psi}(\sF)$ in $\check E$. We say that $\sF$ is {\it non-degenerated} if the closure of its Fourier dual support does not meet the zero section of $\check E$.
\end{definition}

If we replace $\psi$ by $a\psi$ for an element $a\in \bF_p^{\ti}$, then the Fourier dual support of $\sF$ will be replaced by its inverse image by the multiplication by $a$ on $\check E$. In particular, the notion of being non-degenerated dose not depend on $\psi$.

\begin{lemma}[\cite{saito cc} 2.6]\label{ds inclu lemma}
Let $X$ be a $k$-scheme, $\pi:E\ra X$ a vector bundle of constant rank, $s:X\ra E$ the zero section of $\pi$, $\mu:E\ti_X E\ra E$ the addition and $\sF$ and $\sG$ constructible sheaves of $\L$-modules on $E$, where $\sF$ is additive. Let $e\in \G(X,s^*\sF)$ be a section and $u:\sF\bxt\sG\ra \mu^*\sG$ a map such that the composed map
\begin{equation*}
 u|_{s(X)\ti E}\circ(e\ti\id_{\sG}):  \sG\ra s^*\sF\bxt\sG\ra\sG
\end{equation*}
is the identity. Then $\sG$ is additive and the Fourier dual support of $\sG$ is a subset of that of $\sF$.
\end{lemma}

\begin{lemma}[\cite{rc} 3.10]\label{non-deg *!=0}
Let $X$ be a $k$-scheme, $\pi:E\ra X$ a vector bundle of constant rank, and $\sF$ an additive constructible sheaf of $\L$-modules on $E$. If $\sF$ is non-degenerate, $\R\pi_*\sF=\R\pi_!\sF=0$.
\end{lemma}

It follows form \eqref{ftdual}, \eqref{ftprop1} and \eqref{ftprop2} by applying $f$ to the zero section of the dual bundle $\check E$ of $E$ and $K'=\fF_{\psi}(\sF)$.

\subsection{}\label{wtv jvr}
Let $X$ be a connected smooth $k$-scheme of dimension $d$, $D$ a divisor with simple normal crossing on $X$, $\{D_i\}_{i\in I}$ the irreducible components of $D$, $R$ an effective Cartier divisor of $X$ with support in $D$, $U=X-D$ and $V=X-R$ the complementary open subschemes of $D$ and $R$ in $X$ respectively, $j:U\ra X$, $j_V:V\ra X$ and $\nu:U\ra V$ the canonical injections. We denote by $\xxlp$ (resp. $\vvlp$) the framed self-product of $X$ along $D$ (resp. of $V$ along $D\cap V$) (\ref{xstarx}), by $\wt\d:X\ra\xxlp$ the canonical lifting of the diagonal $\d:X\ra\xx$ (\ref{xstarx}) and by $\xxr$ the dilatation of $\xxlp$ along $\wt\d$ of thickening $R$, and we take the notation of (\ref{xxr}). Moreover, we denote by
\begin{equation*}
  \wt\nu:\uu\ra\vvlp\ \ \ {\rm and}\ \ \ j^{(R)}_V:\vvlp\ra\xxr
\end{equation*}
the canonical injections, by $E^{(R)}$ the vector bundle $\xxr\ti_XR$ over $R$ \eqref{vectorbundle xxr}, and by $\check E^{(R)}$ its dual bundle.

Let $\sF$ be a locally constant and constructible sheaf of free $\L$-modules on $U$. We put
\begin{equation*}
\sH=\shom(\pr_{2}^*\sF,\pr_{1}^*\sF)
\end{equation*}
on $\uu$.

\begin{proposition}[\cite{rc} 8.15, 8.17]\label{proof additive rc}
We keep the assumptions and notation of {\rm\ref{wtv jvr}}, moreover, we assume that the ramification of $\sF$ along $D$ is bounded by $R+$ {\rm(\ref{sheaf bound})}. Then,
\begin{itemize}
    \item[(i)]
   $j^{(R)}_*\sH|_{E^{(R)}}$ is additive. Let $S^0_R(\sF)\subset \check E^{(R)}$ be its Fourier dual support.
   \item[(ii)]
   $S^0_R(\sF)$ is the underlying space of a closed subscheme of $\check E^{(R)}$ which is finite over $R$.
  \end{itemize}
\end{proposition}

\begin{proposition}\label{ds inclusion}
We keep the assumptions and notation of {\rm\ref{wtv jvr}}, moreover, we assume that the ramification of $\sF$ along $D$ is bounded by $R+$. Then, for any integer $q\ddy 0$, $\R^q j^{(R)}_{V*}(\wt\nu_*(\sH))|_{E^{(R)}}$ is additive. Let $S^q_R(\sF)\subset \check E^{(R)}$ be the Fourier dual support of $\R^q j^{(R)}_{V*}(\wt\nu_*(\sH))|_{E^{(R)}}$, we have $S^q_R(\sF)\subseteq S^0_R(\sF)$.
\end{proposition}

\begin{proof}
We focus on the situation $q\ddy 1$ since the case $q=0$ is due to \ref{proof additive rc}.
For a scheme $Y$ over $\xx$, we denote by $f_1,f_2:Y\ra X$ the maps induced by the projections $\pr_1,\pr_2:\xx\ra X$, respectively. We denote the fiber product $Y\ti_{f_2,X,f_1}Y$ simply by $Y\ti_XY$.

By (\cite{rc} 5.34, \cite{saito cc} 2.24), there exists a morphism $\lambda:(\xxlp)\ti_X(\xxlp)\ra\xxlp$ that lifts the composed map
$(\xx)\ti_X(\xx)\isora X\ti_kX\ti_kX\xra{\pr_{13}}\xx$, and a smooth morphism $\mu:\xxr\ti_X\xxr\ra\xxr$ that makes the diagram
\begin{equation*}
\xymatrix{\relax
\xxr\ti_X\xxr\ar[r]^-{\mu}\ar[d]&\xxr\ar[d]\\
(\xxlp)\ti_X(\xxlp)\ar[r]^-{\lambda}&\xxlp}
\end{equation*}
commutative, where the vertical arrows are the canonical projections. Moreover, the pull-back of $\mu$ by the canonical injection $E^{(R)}\ra\xxr$
\begin{equation*}
  \mu^{(R)}:E^{(R)}\ti_RE^{(R)}\ra E^{(R)}
\end{equation*}
is the addition of the bundle $E^{(R)}$ (\cite{rc} 5.35). Hence, we have a canonical commutative diagram with Cartesian squares
\begin{equation*}
  \xymatrix{\relax
  (\uu)\ti_X(\uu)=U\ti_kU\ti_kU\ar@{}|{\Box}[rd]\ar[d]_{\ol\nu}\ar[r]^-{\pr_{13}}
  &\uu\ar@/^9mm/[dd]^{j^{(R)}}\ar[d]^{\wt\nu}\\
  (\vvlp)\ti_X(\vvlp)\ar@{}|{\Box}[rd]\ar[r]^-{\mu_V}\ar[d]_{\ol j^{(R)}_V}&\vvlp\ar[d]^{j^{(R)}_V}\\
  \xxr\ti_X\xxr\ar[r]^-{\mu}&\xxr}
\end{equation*}
where $\ol\nu$ and $\ol j^{(R)}_V$ are canonical injections. By adjunction, we have canonical maps
\begin{eqnarray}
  \wt\nu_*(\sH)\bxt^L\wt\nu_*(\sH)&\ra& \ol\nu_*(\sH\bxt\sH), \label{twice kunneth1}\\
  \R j^{(R)}_{V*}(\wt\nu_*(\sH))\bxt^L\R j^{(R)}_{V*}(\wt\nu_*(\sH))&\ra&
  \R \ol j^{(R)}_{V*}(\wt\nu_*(\sH)\bxt^L\wt\nu_*(\sH)).\label{twice kunneth2}
\end{eqnarray}
On $(\uu)\ti_X(\uu)=U\ti_kU\ti_kU$, we have

\begin{equation*}
  \sH\bxt\sH=\shom(\pr_{2}^*\sF,\pr_{1}^*\sF)\ot\shom(\pr_{3}^*\sF,\pr_{2}^*\sF),
\end{equation*}
that gives a map
\begin{equation}\label{hh to p13h}
  \sH\bxt\sH\ra\shom(\pr^*_{3}\sF,\pr^*_{1}\sF)=\pr^*_{13}\sH.
\end{equation}
Since $\mu$ is smooth, by the smooth base change theorem, we have an isomorphism
\begin{equation}\label{twice sbc}
 \mu^*(\R j^{(R)}_{V*}(\wt\nu_*(\sH)))\isora \R\ol j^{(R)}_{V*}(\ol\nu_*(\pr^*_{13}(\sH))).
\end{equation}
The maps \eqref{twice kunneth1}, \eqref{twice kunneth2}, \eqref{hh to p13h} and the inverse of \eqref{twice sbc} induce a map
\begin{equation}\label{pairing rj nu h}
  \R j^{(R)}_{V*}(\wt\nu_*(\sH))\bxt^L\R j^{(R)}_{V*}(\wt\nu_*(\sH))\ra \mu^*(\R j^{(R)}_{V*}(\wt\nu_*(\sH))).
\end{equation}
Consider the following commutative diagram with Cartesian squares
\begin{equation*}
\xymatrix{\relax
U\ti_U(\uu)\ar[r]^-{\d_U\ti\id}\ar[d]_{\wt\nu}\ar@{}|{\Box}[rd]&(\uu)\ti_X(\uu)\ar[d]^{\ol\nu}\\
V\ti_V(\vvlp)\ar[r]^-{\wt\d_V\ti \id}\ar[d]_{j^{(R)}_V}\ar@{}|{\Box}[rd]&(\vvlp)\ti_X(\vvlp)\ar[d]^{\ol j^{(R)}_V}\\
X\ti_X\xxr\ar[r]^-{\d^{(R)}\ti\id}&\xxr\ti_X\xxr}
\end{equation*}
Notice that $ \mu\circ(\d^{(R)}\ti\id)=\id$ (\cite{rc} 5.35). Pulling back \eqref{pairing rj nu h} by $\d^{(R)}\ti \id$, we obtain a commutative diagram
\begin{equation*}
  \xymatrix{\relax
  (\d^{(R)}\ti\id)^*(j^{(R)}_*(\sH)\bxt\R^q j^{(R)}_{V*}(\wt\nu_*(\sH)))\ar[r]\ar[d]_{\theta}&(\d^{(R)}\ti\id)^*\mu^*(\R^q j^{(R)}_{V*}(\wt\nu_*(\sH)))\ar@{=}[d]\\
  j_*\d^*_U(\sH)\bxt\R^q j^{(R)}_V(\wt\nu_*(\sH))\ar[r]^-{\vartheta}&\R^q j^{(R)}_V(\wt\nu_*(\sH))}
\end{equation*}
where $\theta$ is an isomorphism induced by the base change isomorphism \eqref{alpha cf}
\begin{equation*}
  \d^{(R)*}j^{(R)}_*(\sH)\isora j_*\d^*_U(\sH).
\end{equation*}
On $U\ti_U(\uu)=\uu$, we have
\begin{equation*}
  \d^*_U(\sH)\bxt\sH=\shom(\pr_{2}^*\sF,\pr_{2}^*\sF)\ot \shom(\pr_{2}^*\sF,\pr_{1}^*\sF),
\end{equation*}
which induces a map
\begin{equation}\label{dh h to h}
 \d^*_U(\sH)\bxt\sH\ra  \shom(\pr_{2}^*\sF,\pr_{1}^*\sF)=\sH.
\end{equation}
The morphism $\vartheta$ is the following composed map
\begin{equation*}
  j_*\d^*_U(\sH)\bxt\R^q j^{(R)}_V(\wt\nu_*(\sH))\ra \R^q j^{(R)}_{V*}\wt\nu_*(\d^*_U(\sH)\bxt\sH)\ra \R^q j^{(R)}_V(\wt\nu_*(\sH)),
\end{equation*}
where the second arrow is induced by \eqref{dh h to h}. The map
\begin{equation}\label{map e ass id f}
\e:\L\ra j_*\d^*_U(\sH)
\end{equation}
 associated to the element $\id_{\sF}\in\G(X,j_*\d_U^*(\sH))=\End(\sF)$ induces the identity
 \begin{equation*}
   \sH\isora\L\bxt\sH\xra{\e|_U\ti\id}\d_U^*(\sH)\bxt\sH\xra{\eqref{dh h to h}}\sH.
 \end{equation*}
Hence $\e$ and $\vartheta$ induce the identity of $\R^q j^{(R)}_V(\wt\nu_*(\sH))$.
Restrict \eqref{pairing rj nu h} to $E^{(R)}\ti_RE^{(R)}$, we obtain a map
\begin{equation}\label{pairing res Er}
  (j^{(R)}_*(\sH)|_{E^{(R)}})\bxt(\R^q j^{(R)}_{V*}(\wt\nu_*(\sH))|_{E^{(R)}})\ra \mu^{(R)*}(\R^q j^{(R)}_{V*}(\wt\nu_*(\sH))|_{E^{(R)}}).
\end{equation}
Notice that the zero section $s^{(R)}:R\ra E^{(R)}$ is just the pull-back of $\d^{(R)}:X\ra\xxr$ by $E^{(R)}\ra\xxr$. After restricting \eqref{pairing res Er} to $s^{(R)}(R)\ti_RE^{(R)}$, the map $\e|_R$ \eqref{map e ass id f} induces the identity of $\R^q j^{(R)}_{V*}(\wt\nu_*(\sH))|_{E^{(R)}}$.
Hence, the proposition follows from (\ref{ds inclu lemma}) (applied with $\sF=j^{(R)}_*(\sH)|_{E^{(R)}}$ and $\sG=\R^q j^{(R)}_{V*}(\wt\nu_*(\sH))|_{E^{(R)}}$).
\end{proof}

\begin{definition}[\cite{rc} 8.23]\label{clean isocline sheaf}
We keep the assumptions and notation of {\rm\ref{wtv jvr}}, moreover, we assume that the conductor of $\sF$ relatively to $X$ is the effective divisor $R$ (\ref{bounded conductor}) and that $\sF$ is isoclinic along $D$ (\ref{isoclinic}). We say that $\sF$ is {\it clean} along $D$ if the following conditions are satisfied:
\begin{itemize}
  \item[(i)]
  the ramification of $\sF$ along $D$ is bounded by $R+$;
  \item[(ii)]
  the additive sheaf $j^{(R)}_*\sH|_{E^{(R)}}$ on $E^{(R)}$ is non-degenerated (\ref{fourier ds}).
\end{itemize}
\end{definition}

\begin{proposition}\label{saito cc3.4}
We keep the assumptions and notation of {\rm\ref{wtv jvr}}, moreover, we assume that the conductor of $\sF$ relatively to $X$ is the effective divisor $R$ and that $\sF$ is isoclinic and clean along $D$ {\rm(\ref{clean isocline sheaf})}. Then, we have
  \begin{equation*}
    \R\G_{E^{(R)}}(\xxr,j^{(R)}_*(\sH)(d))=0.
  \end{equation*}
\end{proposition}
\begin{proof}
We denote by $i:E^{(R)}\ra\xxr$ the canonical injection and $\pi:E^{(R)}\ra R$ the canonical projection. Notice that (\cite{rc} 5.26)
\begin{equation*}
  \vvlp=\xxr\ti_XV=\xxr-E^{(R)},
\end{equation*}
then
\begin{equation*}
  \R^q i^! (j^{(R)}_*\sH)=\bigg{\{}\begin{array}{ll}
    0&{\rm when}\ \ \ q\xdy 1;\\
    i^*\R^{q-1}j^{(R)}_{V*}(\wt\nu_*(\sH))& {\rm when}\ \ \ q\ddy 2.
  \end{array}
\end{equation*}
Since $\sF$ is clean along $D$, for any integer $q$, the sheaf $i^*\R^{q-1}j^{(R)}_{V*}(\nu_*(\sH))$ on $E^{(R)}$ is additive and non-degenerated (\ref{ds inclusion}). Hence, for any integer $q$, $\R\pi_*\R^q i^! (j^{(R)}_*(\sH))=0$ (\ref{non-deg *!=0}).
Hence,
\begin{equation*}
  \R\G_{E^{(R)}}(\xxr,j^{(R)}_*(\sH)(d))=\R\G(R,\R\pi_*\R i^! (j^{(R)}_*(\sH))(d))=0.
\end{equation*}
\end{proof}

\begin{remark}
Proposition \ref{saito cc3.4} is used in the proof of (\cite{saito cc} 3.4). However, the proof of loc. cit. relies on (\cite{saito cc} 2.25) which is not enough. We reinforce it in \ref{ds inclusion}.
\end{remark}

\subsection{}\label{charcycle}
We keep the assumptions and notation of {\rm\ref{wtv jvr}} and we denote by
\begin{equation*}
\mathrm{T}^*X(\log D)=\mathbf{V}(\O^1_{X/k}(\log D)^{\vee}),
\end{equation*}
the logarithmic cotangent bundle of $X$, by $\sigma:X\ra \mathrm{T}^*X(\log D)$ the zero section, for $i\in I$, by $\xi_i$ the generic point of $D_i$, by $F_i$ the residue field of $\sO_{X,\xi_i}$, by $S_i=\spec(\sO_{K_i})$ the henselization of $X$ at $\xi_i$, by $\eta_i=\spec(K_i)$ the generic point of $S_i$, by $\ol K_i$ a separable closure of $K_i$ and by $G_i$ the Galois group $\mathrm{Gal}(\ol K_i/K_i)$.

We assume moreover that the conductor of $\sF$ is $R$, and that $\sF$ is isoclinic and clean along $D$. We denote by $M_i$ the $\L[G_i]$-module corresponding to $\sF|_{\eta_i}$. Since $\sF$ is isoclinic along $D$, $M_i$ has just one slope $r_i$. We put $I_{\rw}=\{i\in I; r_i>0\}$ and $S=\sum_{i\in I_{\rw}}D_i$.
For $i\in I_{\rw}$, let
\begin{equation*}
  M_i=\oplus_{\chi}M_{i,\chi}
\end{equation*}
be the central character decomposition of $M_i$ (\ref{slope center decomp cf}). Note that $M_{i,\chi}$ is a free $\L$-module of finite type for all $\chi$. By enlarging $\L$, we may assume that for all central characters $\chi$ of $M_i$, we have $\L_{\chi}=\L$. Since $\mathrm{Gr}^{r_i}_{\log}G_i$ is abelian and killed by $p$ (\cite{saito cc} 1.24), each $\chi$ factors uniquely as
$\mathrm{Gr}^{r_i}_{\log}G_i\ra\mathbb{F}_p\xra{\psi}\L^{\ti}$, where $\psi$ is the non-trivial additive character fixed in \ref{fixnotation}. We denote also by $\chi$ the induced character and by
\begin{equation*}
  \mathrm{rsw}(\chi):\mathfrak{m}^{r_i}_{\ol K_i}/\mathfrak{m}^{r_i+}_{\ol K_i}\ra\O^1_{F_i}(\log)\ot_{F_i}\ol F_i
\end{equation*}
its refined Swan conductor \eqref{rsw llcc} (where the notation are defined as in \ref{rsw intro cf}). Let $F_{\chi}$ be the field of definition of $\mathrm{rsw}(\chi)$, which is a finite extension of $F_i$ contained in $\ol F_i$. The refined Swan conductor $\mathrm{rsw}(\chi)$ defines a line $L_{\chi}$ in $\mathrm{T}^*X(\log D)\ot_X\spec(F_{\chi})$. Let $\ol L_{\chi}$ be the closure of the image of $L_{\chi}$ in $\mathrm{T}^*X(\log D)$. For $i\in I_{\rw}$, we put
\begin{equation}\label{CCi}
  CC_i(\sF)=\sum_{\chi}\frac{r_i\cdot\rk_{\L}(M_{i,\chi})}{[F_{\chi}:F_i]}[\ol L_{\chi}],
\end{equation}
which is a $d$-cycle on $\mathrm{T}^*X(\log D)\ti_XD_i$. We define a $d$-cycle $CC^*(\sF)$ on $\mathrm{T}^*X(\log D)\ti_XS$ by
\begin{equation}\label{CC0}
 CC^*(\sF)=\sum_{i\in I_{\rw}}CC_i(\sF).
\end{equation}
We define the {\em characteristic cycle} of $\sF$ and denote by $CC(\sF)$, the $d$-cycle on $\mathrm{T}^*X(\log D)$ defined by (\cite{saito cc} 3.6)
\begin{equation*}
  CC(\sF)=(-1)^d\left(\rk_{\L}(\sF)[\sigma(X)]+CC^*(\sF)\right).
\end{equation*}

\section{Tsushima's refined characteristic class}

\subsection{}\label{notllcc}
In this section, $X$ denotes a connected smooth $k$-scheme of dimension $d$, $D$ a divisor with simple normal crossing on $X$ and $\{D_i\}_{i\in I}$ the irreducible components of $D$. We assume that $I=I_{\rmt}\coprod I_{\rw}$, and we put $S=\bigcup_{i\in I_{\rw}}D_i$, $T=\bigcup_{i\in I_{\rmt}}D_i$, $U=X-D$ and $V=X-S$. We denote by $j:U\ra X$, $j_V:V\ra X$ and $\nu:U\ra V$ the canonical injections.

We denote by $\xxlb$ the blow-up of $\xx$ along $\{D_i\ti_kD_i\}_{i\in I}$, by $\xxd$ the blow-up of $\xx$ along $\{D_i\ti_kD_i\}_{i\in I_{\rmt}}$, by $X\rtimes_kX$ the left-framed self-product of $X$ along $D$ and by $\xxlp$ the framed self-product of $X$ along $D$ (\ref{xstarx}).
For any open subschemes $Y$ and $Z$ of $X$, we put
\begin{eqnarray*}
  \yzlb&=&(\yz)\ti_{(\xx)}\xxlb,\\
  \yzd&=&(\yz)\ti_{(\xx)}\xxd,\\
  \yzlp&=&(\yz)\ti_{(\xx)}(\xxlp).
\end{eqnarray*}
Notice that $\vvlb=(V\ti_kV)^{\dagger}$ and $\vvlp=\prod_{i\in I_{\rmt}}((\xxlp)_i\ti_{(\xx)}(\vv))$. We have the
following commutative diagram with Cartesian squares
\begin{equation}\label{diag5ti2}
  \xymatrix{\relax
  \uu\ar@{=}[d]\ar@{}[rd]|{\Box}\ar@{=}[r]&\uu\ar@{}[rd]|{\Box}\ar[r]^-{\wt\nu}\ar[d]^-{\nu_2}&\vvlp\ar[d]^{\varphi_2}\ar@{=}[r]\ar@{}[rdd]|{\Box}&\vvlp\ar[r]\ar[dd] \ar@{}[rdd]|{\Box} &\xxlp\ar[dd]^{\varphi} \\
  \uu\ar@{=}[d]\ar@{}[rd]|{\Box}\ar[r]&\uv\ar@{}[rd]|{\Box}\ar@{=}[d]\ar[r]^-{\nu^{\rtimes}_1}&V\rtimes_kV\ar[d]^{\varphi_1}& & \\
  \uu\ar@{=}[d]\ar@{}[rd]|{\Box}\ar[r]&\uv\ar@{}[rd]|{\Box}\ar@{=}[d]\ar[r]^-{\nu^{\dagger}_1} &\vvd\ar@{}[rd]|{\Box}\ar[d]^h\ar[r]^-{j^{\dagger}_2} &(V\ti_kX)^{\dagger}\ar@{}[rd]|{\Box}\ar[r]^-{j^{\dagger}_1}\ar[d]^-g &\xxd\ar[d]^-f\\
  \uu\ar[r]_-{\nu_2}&\uv\ar[r]_-{\nu_1} &\vv\ar[r]_-{j_2} &\vx\ar[r]_-{j_1} &\xx }
\end{equation}
where all horizontal arrows are open immersions. We denote by $\wt j:\uu\ra\xxlp$ the canonical injection.

We denote by $\d:X\ra\xx$ the diagonal map. By the universality of the blow-up, $\d$ induces closed immersions
\begin{equation}\label{ddwtd}
   \d^{\dagger}: X\ra\xxd\ \ \ {\rm and}\ \ \ \wt\d:X\ra\xxlp,
   \end{equation}
and hence, by pull-back, the following closed immersions
\begin{equation}\label{dvdwtdv}
 \d^{\dagger}_V:V\ra (V\ti_kV)^{\dagger} \ \ {\rm and}\ \ \wt\d_V:V\ra\vvlp.
\end{equation}

\subsection{}\label{sheaf setting not}
In the following of this section,  $\sF$ denotes a locally constant and constructible
sheaf of free $\L$-modules on $U$, tamely ramified along $T\cap V$ relatively to $V$.
We put
\begin{eqnarray*}
\sH_0&=&\shom(\pr^*_2\sF,\pr^*_1\sF)\ \ \ {\rm on}\ \ \ \uu,\\
\sH&=&\rshom(\pr_2^*\sF,\R\pr_1^!\sF)\ \ \ {\rm on}\ \ \ \uu,\\
{\ol\sH}&=&\rshom(\pr_2^*j_{!}\sF,\R\pr_1^!j_{!}\sF)\ \ \ {\rm on}\ \ \ \xx,\\
\wt\sH&=&\wt j_*\sH_0(d)[2d]\ \ \ {\rm on}\ \ \ \xxlp.
\end{eqnarray*}
We have a canonical isomorphism $\sH\isora\sH_0(d)[2d]$.

\subsection{}\label{some sHv}
We denote by $\ol\sH_V$ the restriction of $\ol\sH$ to $\vv$ and by $\wt\sH_V$ the restrictions of $\wt\sH$ to $\vvlp$. Notice that $\ol\sH_V\isora\R\shom(\pr^*_2\nu_!\sF,\R\pr_2^!\nu_!\sF)$.
We put
\begin{equation*}
 \ol\sH^{\dagger}_V=\varphi_{1!}(\R\varphi_{2*}(\wt\sH_V))\ \ \ {\rm on}\ \ \ \vvd.
\end{equation*}
Since $\nu_1$ is an open immersion, the base change maps give by composition an isomorphism \eqref{diag5ti2}
\begin{equation*}
  \nu^{\dagger*}_1(\ol\sH^{\dagger}_V)=\nu^{\dagger*}_1\varphi_{1!}(\R\varphi_{2*}(\wt\sH_V))
  \isora\nu^{\rtimes*}_1(\R\varphi_{2*}(\wt\sH_V))\isora\R\nu_{2*}\sH.
\end{equation*}
By \eqref{kunforH}, we have
\begin{equation*}
  h^*(\ol\sH_V)\isora h^*\nu_{1!}\R \nu_{2*}\sH\isora \nu^{\dagger}_{1!}(\R \nu_{2*}\sH)\isora \nu^{\dagger}_{1!}(\nu^{\dagger*}_1(\ol\sH^{\dagger}_V)).
\end{equation*}
It induces a canonical map
\begin{equation}\label{jussmap}
  h^*(\ol\sH_V)\ra \ol\sH^{\dagger}_V,
\end{equation}
that extends the identity of $\sH$ on $\uu$. Since $\sF$ is tamely ramified along the divisor $T\cap V$ relatively to $V$, the adjoint map
\begin{equation*}
\ol\sH_V\ra\R h_*(\ol\sH^{\dagger}_V)
\end{equation*}
is an isomorphism by (\cite{as} 2.2.4).

\subsection{}\label{alphaid}
We put
\begin{equation*}
  \ol\sH^{\dagger}=j^{\dagger}_{1!}(\R j^{\dagger}_{2*}(\ol\sH^{\dagger}_V))
\end{equation*}
on $\xxd$, and we consider the following composition of maps
\begin{equation}\label{defllcckey}
  f^*\ol\sH=f^*j_{1!}\R j_{2*}(\ol\sH_V)\isora j^{\dagger}_{1!}g^*\R j_{2*}(\ol\sH_V)\ra j^{\dagger}_{1!}\R j^{\dagger}_{2*}h^*(\ol\sH_V)\xra{\eqref{jussmap}}j^{\dagger}_{1!}(\R j^{\dagger}_{2*}(\ol\sH^{\dagger}_V))=\ol\sH^{\dagger},
\end{equation}
where the second and the third arrows are induced by the base change maps.
\begin{lemma}[\cite{ts2} Lemma 3.13]
The adjoint map of \eqref{defllcckey}
\begin{equation*}
  \ol\sH\ra\R f_*\ol\sH^{\dagger}
\end{equation*}
is an isomorphism.
\end{lemma}

\subsection{}\label{sLdagger}
We put $X^{\dagger}=f^{-1}(\d(X))$ and $S^{\dagger}=f^{-1}(\d(S))$. We denote by $\gamma^{\dagger}:\xxd\backslash X^{\dagger}\ra\xxd$ the canonical injection, which is an open immersion and put
\begin{equation}\label{sLdagger'}
\sL^{\dagger}=\R\g^{\dagger}_*(\L).
\end{equation}
The map \eqref{defllcckey} induces by pull-back a map
\begin{equation}\label{fupstar}
  \h^0_X(\xx,\ol\sH)\ra\h^0_{X^{\dagger}}(\xxd,\ol\sH^{\dagger}).
\end{equation}
The adjunction $\L\ra \sL^{\dagger}$ induces a map
\begin{equation}\label{adjg'}
\h^0_{X^{\dagger}}(\xxd,\ol\sH^{\dagger})\ra\h^0_{X^{\dagger}}(\xxd,\ol\sH^{\dagger}\ot^L \sL^{\dagger}).
\end{equation}
By \eqref{gpd}, we have a canonical isomorphism $\End(j_!\sF)\isora\h^0_X(\xx,\ol\sH)$. We denote also $\id_{j_!\sF}$ the image of $\id_{j_!\sF}\in\End(j_!\sF)$ in $\h^0_X(\xx,\ol\sH)$. Its image in
$\h^0_{X^{\dagger}}(\xxd,\ol\sH^{\dagger}\ot^L\sL^{\dagger})$ by the composition of the maps \eqref{fupstar} and \eqref{adjg'} will be denoted by
$\alpha(j_!\sF)$.

\begin{proposition}[\cite{ts2}, 3.14 and 3.15]\label{injkeylemma}
The canonical map
\begin{equation}\label{injkey}
  \h^0_{S^{\dagger}}(\xxd,\ol\sH^{\dagger}\ot^L\sL^{\dagger})\ra \h^0_{X^{\dagger}}(\xxd,\ol\sH^{\dagger}\ot^L\sL^{\dagger})
\end{equation}
is injective and there exists a unique element
 \begin{equation}\label{alpha0}
  \alpha_0(j_!\sF)\in \h^0_{S^{\dagger}}(\xxd,\ol\sH^{\dagger}\ot^L\sL^{\dagger})
 \end{equation}
whose image by \eqref{injkey} is $\alpha(j_!\sF)$.
\end{proposition}

\subsection{}
The squares of the following commutative diagram
\begin{equation*}
  \xymatrix{\relax
  X\ar@{}[rd]|{\Box}\ar[d]_{\d'}&V\ar@{}[rd]|{\Box}\ar[l]_{j_V}\ar[d]&V\ar@{}[rd]|{\Box}\ar@{=}[l]\ar[d]^{\d_V^{\dagger}}&V\ar@{}[rd]|{\Box}\ar@{=}[l]\ar[d]^{\wt\d_V}&U\ar[l]_-{\nu}\ar[d]^{\d_U}\\
  \xxd&(V\ti_k X)^{\dagger}\ar[l]_-{j^{\dagger}_1}&\vvd\ar[l]_-{j^{\dagger}_2}&\vvlp\ar[l]_-{\varphi_1\circ\varphi_2}&\uu\ar[l]_-{\wt\nu}}
\end{equation*}
are Cartesian and all the horizontal arrows are open immersions. By (\ref{sheaf bound}) and (\ref{tame ram}), since $\sF$ is tamely ramified along $T\cap V$ relatively to $V$, we have an isomorphism
\begin{equation}\label{tame iso}
  \wt\d_V^*(\wt\nu_*(\sH_0))\isora \nu_*(\d_U^*(\sH_0)).
\end{equation}
The base change maps give by composition an isomorphism
\begin{equation}\label{locev1}
  \d^{\dagger *}(\ol\sH^{\dagger})\isora j_{V!}\d_V^{\dagger *}(\ol\sH^{\dagger}_V)\isora j_{V!}\wt\d_V^*(\wt\sH_V)\isora j_{V!} \nu_*(\send(\sF))(d)[2d],
\end{equation}
where the third arrow is \eqref{tame iso}. There exists a unique map
\begin{equation}\label{locev2}
  \mathrm{Tr}_V:\nu_*(\send(\sF))\ra \L_{V}
\end{equation}
which extends the trace map $\mathrm{Tr}:\send(\sF)\ra \L_U$ (\cite{as} (2.9)). The maps  \eqref{locev1} and \eqref{locev2} give an evaluation map
\begin{equation}
\ev^{\dagger}:\d^{\dagger *}(\ol\sH^{\dagger})\ra j_{V!}(\sK_{V}).
\end{equation}
Composing with the canonical map $j_{V!}(\sK_V)\ra\sK_X$, we obtain a morphism
\begin{equation}\label{ev'}
  \h^0_{S}(X,\d^{\dagger *}(\ol\sH^{\dagger})\ot^L\d^{\dagger *}\sL^{\dagger})\ra\h^0_S(X, \sK_X\ot^L\d^{\dagger *}\sL^{\dagger}).
\end{equation}
The pull-back by $\d^{\dagger}$ gives a morphism
\begin{equation}\label{pullback d'}
\h^0_{S^{\dagger}}(\xxd,\ol\sH^{\dagger}\ot^L \sL^{\dagger})\ra \h^0_{S}(X,\d^{\dagger *}(\ol\sH^{\dagger})\ot^L\d^{\dagger *}\sL^{\dagger}).
\end{equation}
Composing \eqref{ev'} and \eqref{pullback d'}, we get a map
\begin{equation}\label{1,2,3}
  \h^0_{S^{\dagger}}(\xxd,\ol\sH^{\dagger}\ot^L \sL^{\dagger})\ra \h^0_S(X, \sK_X\ot^L\d^{\dagger *}\sL^{\dagger}).
\end{equation}

\begin{lemma}[\cite{ts1} Lemma 2.3]\label{isoadjunction}
The canonical map
\begin{equation}\label{isoadjdg}
  \h^0_{S}(X,\sK_X)\ra H^0_{S}(X,\sK_X\ot^L\d^{\dagger *}\sL^{\dagger})
\end{equation}
induced by the canonical map $\L\ra\d^{\dagger *}\sL^{\dagger}$, is an isomorphism.
\end{lemma}

\subsection{}\label{defllcc}
Composing \eqref{1,2,3} and the inverse of \eqref{isoadjdg}, we get a map
\begin{equation}\label{kappa}
  \kappa:\h^0_{S^{\dagger}}(\xxd,\ol\sH^{\dagger}\ot^L \sL^{\dagger})\ra\h^0_S(X, \sK_X).
\end{equation}
We call $\kappa(\alpha_0(j_!\sF))\in\h^0_{S}(X,\sK_X)$ the {\it refined characteristic cycle} of $j_!\sF$, and we denote it by $C_{S}(j_!\sF)$.
\begin{remark}
If $T=\emptyset$, we have $(\xx)^{\dagger}=\xx$, $(\vv)^{\dagger}=\vvlp=\uu$, $X^{\dagger}=X$, $S^{\dagger}=S$ and, by \eqref{kunforH}, $\ol\sH^{\dagger}=\ol\sH$. It is easy to see that (\ref{lcc as})
\begin{equation*}
  C_S(j_!(\sF))=C_S^0(j_!(\sF))\in\h^0_{S}(X,\sK_X).
\end{equation*}

\end{remark}

\subsection{}\label{X to Y not}
In the following of this section, $Y$ denotes a connected smooth $k$-scheme, $Z$ a closed subscheme of $Y$ and $W=Y-Z$ the complementary open subscheme of $Z$ in $Y$. We assume that there exists a proper flat morphism $\pi:X\ra Y$ such that $V=\pi^{-1}(W)$, $Q=\pi^{-1}(Z)$ is a divisor with normal crossing, that $S=Q_{\mathrm{red}}$, that the canonical projection $\pi_V:V\ra W$ is smooth and that $T\cap V$ is a divisor with simple normal crossing relatively to $W$. We have a commutative diagram with Cartesian squares
\begin{equation*}
  \xymatrix{\relax
  U\ar[r]^-{\nu}\ar[rd]_{\pi_U}&V\ar@{}|{\Box}[rd]\ar[r]^{j_V}
  \ar[d]^{\pi_V}&X\ar@{}|{\Box}[rd]\ar[d]^{\pi}&Q\ar[l]_-{i_Q}\ar[d]^-{\pi_Q}\\
  & W\ar[r]^{j_W}&Y&Z\ar[l]_-{i_Z}}
\end{equation*}

We make the following remarks:
\begin{itemize}
\item[(i)]
 For any locally constant and constructible sheaf of $\L$-modules $\sG$ tamely ramified along the divisor $T\cap V$ relatively to $V$, $\pi_V$ is universally locally acyclic relatively to $\nu_!(\sG)$ (\cite{sga4 1/2} Appendice to Th. Finitude, \cite{wrcb} 3.14). Since $\pi_V$ is proper, all cohomology groups of $\R\pi_{U!}(\sG)$ are locally constant and constructible on $W$.
\item[(ii)]
Since $\pi$ is proper, we have a push-forward
\begin{equation}\label{proper pushf}
  \h^0_S(X,\sK_X)\isora\h^0_Q(X,\sK_X)\ra\h^0_Z(Y,\sK_Y)
\end{equation}
defined by applying the functor $\h^0(Z,-)$ to the following composed map
\begin{equation*}
\R\pi_{Q*}(\sK_Q)\isora\R\pi_{Q*}\R\pi_Q^!(\sK_Z)\isora \R\pi_{Q!}\R \pi^!_Q(\sK_Z)\ra \sK_Z,
\end{equation*}
where the third arrow is induced by the adjunction.
\end{itemize}

\begin{theorem}[Localized Lefschetz-Verdier trace formula, \cite{ts2} 5.4]\label{loc lef tr}
We have {\rm(\ref{lcc as}, \ref{defllcc})}
\begin{equation*}
  \pi_*(C_S(j_!(\sF)))=C^0_Z(j_{W!}(\R\pi_{U!}(\sF)))
\end{equation*}
in $\h^0_Z(Y,\sK_Y)$.
\end{theorem}

\subsection{}\label{as formula}
We assume that $Y$ is of dimension $1$ and that $Z$ is a closed point $y$ of $Y$, and we denote by $\ol y$ a geometric point of $Y$ localized at $y$, by $Y_{(\ol y)}$ the strict localization of $Y$ at $\ol y$, by $\ol\eta$ a geometric generic point of $Y_{(\ol y)}$. For any object $\sG$ of $\dctf(W,\L)$ with locally constant cohomology groups. We put (\cite{sga4 1/2} Rapport 4.4)
\begin{eqnarray*}
\rk_{\L}(\sG_{\ol\eta})&=&\mathrm{Tr}(\id;\sG_{\ol\eta}),\\
\mathrm{sw}_y(\sG_{\ol\eta})&=&\sum_{q\in\bZ}(-1)^q\mathrm{sw}_y((\sH^q(\sG))_{\ol\eta}),\\
\mathrm{dimtot}_y(\sG_{\ol\eta})&=&\rk_{\L}(\sG_{\ol\eta})+\mathrm{sw}_y(\sG_{\ol\eta}),
\end{eqnarray*}
where $\mathrm{sw}_y((\sH^q(\sG))_{\ol\eta})$ denotes the Swan conductor of $(\sH^q(\sG))_{\ol\eta}$ at $y$ (\cite{serre gr} 19.3).
By (\cite{ts2} 4.1), we have
\begin{equation}\label{lcc to sw}
 C^0_{\{y\}}(j_{W!}(\sG))-\rk_{\L}(\sG|_{\ol\eta})\cdot C^0_{\{y\}}(j_{W!}(\L_W)) =-\mathrm{sw}_{y}(\sG_{\ol\eta})
\end{equation}
in $\h^0_{\{y\}}(Y,\sK_Y)\isora\L$. In fact, the proof of \eqref{lcc to sw} is simpler than the general case treated in (\cite{ts2} 4.1), since $Y$ is of dimension $1$, we can use the usual Swan conductor rather than the generalized one (\cite{ksann} 4.2.2).

\begin{corollary}[\cite{ts2} 5.5]\label{llcc ns cor}
Keep the notation and assumptions of {\rm\ref{as formula}}. We have
\begin{equation}\label{llcc ns}
  \mathrm{sw}_y(\R\G_c(U_{\ol\eta},\sF|_{U_{\ol\eta}}))-\rk_{\L}(\sF)
  \cdot\mathrm{sw}_y(\R\G_c(U_{\ol\eta},\L))=-\pi_*(C_S(j_!(\sF))-\rk_{\L}(\sF)\cdot C_S(j_!(\L_U)))
  \end{equation}
in $\h^0_{\{y\}}(Y,\sK_Y)\isora\L$.
\end{corollary}
\begin{proof}

Since $\sF$ is tamely ramified along $T\cap V$ relatively to $V$, we have (\cite{illusie} 2.7)
\begin{equation*}
\rk_{\L}(\R\pi_{U!}(\sF)|_{\ol\eta})=\rk_{\L}(\sF)\cdot\rk_{\L}(\R\pi_{U!}(\L_U)|_{\ol\eta}).
\end{equation*}
Then, by \eqref{loc lef tr} and \eqref{lcc to sw},
\begin{eqnarray}
&&\pi_*(C_S(j_!(\sF))-\rk_{\L}(\sF)\cdot C_S(j_!(\L_U))) \nonumber\\
&=&C^0_{\{y\}}(j_{W!}\R\pi_{U!}(\sF))-\rk_{\L}(\sF)\cdot C^0_{\{y\}}(j_{W!}\R\pi_{U!}(\L_U))\nonumber\\
&=&C^0_{\{y\}}(j_{W!}\R\pi_{U!}(\sF))-\rk_{\L}(\R\pi_{U!}(\sF)|_{\ol\eta})\cdot C^0_{\{y\}}(j_{W!}(\L_W)) \nonumber    \\
&&-\rk_{\L}(\sF)\cdot\left(C^0_{\{y\}}(j_{W!}\R\pi_{U!}(\L_U))-\rk_{\L}(\R\pi_{U!}(\L)|_{\ol\eta})\cdot C^0_{\{y\}}(j_{W!}(\L_W))\right)\nonumber\\
&=&-\mathrm{sw}_{y}(\R\pi_{U!}(\sF)|_{\ol\eta})+
\rk_{\L}(\sF)\cdot\mathrm{sw}_{y}(\R\pi_{U!}(\L)|_{\ol\eta}). \nonumber
\end{eqnarray}
By the proper base change theorem \eqref{pbc}, we have
\begin{equation*}
  \R\pi_{U!}(\sF)|_{\ol\eta}\isora\R\G_c(U_{\ol\eta},\sF|_{U_{\ol\eta}})
  \ \ \ {\rm and}\ \ \   \R\pi_{U!}(\L)|_{\ol\eta}\isora\R\G_c(U_{\ol\eta},\L).
\end{equation*}
Then \eqref{llcc ns} follows.
\end{proof}

\section{The conductor formula}
\subsection{}\label{notation cul}
In this section, we take again the assumptions of \ref{notllcc} and \ref{sheaf setting not} and we will take the notation introduced in $\S 7$. Let $R$ be the conductor of $\sF$ \eqref{sheaf bound} that we assume having integral coefficients of support in $S$. We assume also that $\sF$ is isoclinic and clean along $D$ (\ref{isoclinic} and \ref{clean isocline sheaf}). Notice that if $R=0$, a sheaf $\sF$ is tamely ramified along $D$ and is automatically isoclinic and clean.

\begin{theorem}\label{keytheorem llcc lchern}
Let $\sF$ be a sheaf on $U$ as in {\rm \ref{notation cul}}.
Assume that $T\cap S=\emptyset$ or that $\rk_{\L}(\sF)=1$. Then, we have \eqref{locchern}
\begin{eqnarray*}
  &&C_{S}(j_!(\sF))-\rk_{\L}(\sF)\cdot C_{S}(j_!(\L_U))\\
  &=&(-1)^{d}\rk_{\L}(\sF)\cdot c_d\left(\OX(\log D)\ot_{\sO_X}\sO_X(R)-\OX(\log D)\right)^X_{S}\cap[X]\in \h^0_{S}(X,\sK_X),\nonumber
\end{eqnarray*}
where the right hand side is considered as an element of $\h^0_{S}(X,\sK_X)$ by the cycle map \eqref{cycle map cl}.
\end{theorem}

The theorem will be proved in \ref{proof keytheorem} after some preliminaries. We will deduce from it the theorem \ref{cond form intro} in \ref{proof cond form}. The case where $\rk_{\L}(\sF)=1$ is due to Tsushima (\cite{ts2} 5.9).

\begin{remark}
 Although we follow the same lines as \cite{ts1} for sheaves of higher ranks, the situation is technically more involved.
 The assumption $S\cap T=\emptyset$ is required for the injectivity of a map $\lambda$ defined in \eqref{lambda focus}, which is a crucial step in my proof (cf. \ref{theorem need 1}). We don't know if it holds without this assumption.
\end{remark}

\subsection{}
We consider $\xxlp$ as an $X$-scheme by the second projection, and we denote by $\xxr$ the dilatation of $\xxlp$ along $\wt\d$ of thickening $R$ \eqref{dilatation}. We have a Cartesian diagram \eqref{xxrcar}
\begin{equation*}
  \xymatrix{\relax
  U\ar@{}[rd]|{\Box}\ar[r]^-{\d_U}\ar[d]_j&\uu\ar[d]^{j^{(R)}}\\
  X\ar[r]^-{\d^{(R)}}&\xxr}
\end{equation*}
We denote by $\fr:\xxr\ra\xx$ and $\varphi^{(R)}:\xxr\ra\xxd$ the canonical projections. We put $X^{(R)}=f^{(R)-1}(\d(X))$ and $S^{(R)}=f^{(R)-1}(\d(S))$.

We put
\begin{equation*}
\sHr=\jr_*(\sH_0)(d)[2d]
\end{equation*}
on $\xxr$. Notice that $\sHr|_{\vvlp}=\wt\sH_V$ (\ref{some sHv}).

\begin{proposition}[\cite{saito cc} Corollary 3.3]\label{jussmapR}
There exists a unique homomorphism {\rm(\ref{sheaf setting not})}
\begin{equation}\label{fR}
  f^{(R)*}(\ol\sH)\ra\sHr
\end{equation}
extending the identity of $\sH$ on $\uu$.
\end{proposition}

\subsection{}
We deduce from \eqref{fR} by pull-back a map
\begin{equation}\label{fRupstar}
  \h^0_X(\xx,\ol\sH)\ra\h^0_{\xr}(\xxr,\sHr).
\end{equation}
We put (\ref{sLdagger'})
\begin{equation}\label{abbr Rg}
 \sL^{(R)}=\varphi^{(R)*}(\sL^{\dagger}).
\end{equation}
The canonical map $\L\ra\sL^{(R)}$ induces a map
\begin{equation}\label{HR ot adjunction}
  \h^0_{\xr}(\xxr,\sHr)\ra  \h^0_{\xr}(\xxr,\sHr\ot^L\sL^{(R)}).
\end{equation}
The canonical injection $\sr\ra\xr$ induces a map
\begin{equation}\label{lambda focus}
\lambda: \h^0_{\sr}(\xxr,\sHr\ot^L\sL^{(R)})\ra \h^0_{\xr}(\xxr,\sHr\ot^L\sL^{(R)}).
\end{equation}

\subsection{}
We denote by $\wt V=\xr-\sr$ the complementary open subscheme of $\sr$ in $\xr$,
by $\wt\i_V:\wt V\ra \vvlp$, $\wt \gamma_V:(\vvlp)\backslash \wt V\ra\vvlp$ and $\gamma_U:(\uu)\backslash \d_U(U)\ra \uu$ the canonical injections. We put $D^{(R)}=f^{(R)-1}(\d(D))$ and $T^{(R)}=f^{(R)-1}(\d(T))$. Notice that $T^{(R)}\cup S^{(R)}=D^{(R)}$ and that $\d_U(U)$ is the complementary open subscheme of $D^{(R)}$ in $\xr$. We put
\begin{equation*}
  \wt\sL_V=\R\wt\g_{V*}(\L).
\end{equation*}
Notice that $\wt\sL_V\isora\sL^{\dagger}|_{\vvlp}=\sL^{(R)}|_{\vvlp}$ \eqref{sLdagger'} and \eqref{abbr Rg}.

\begin{proposition}[\cite{ts1}  2.2]\label{lambda rank 1 iso}
If the sheaf $\sF$ on $U$ has rank $1$, the map $\lambda$ \eqref{lambda focus} is an isomorphism.
\end{proposition}
\begin{proof}
 It is sufficient to show that, for any integer $q$, $\h^{q}_{\wt V}(\vvlp,\wt\sH_V\ot^L\wt\sL_V)=0$. Since $\sF$ is of rank $1$ and is tamely ramified along $T\cap V$ relatively to $V$, $\wt\nu_*\sH_0$ is a locally constant and constructible sheaf on $\vvlp$ (\cite{as} 4.2.2.1).
By (\cite{fu} 6.5.5), we have
 \begin{equation*}
  \wt\sH_V\ot^L\wt\sL_V\isora \R\wt\g_{V*}(\wt\g_V^*(\wt\sH_V)).
 \end{equation*}
Since $\R\wt\i^!\R\wt\g_*=0$ \eqref{low*up!}, for any integer $q$,
\begin{equation*}
\h^{q}_{\wt V}(\vvlp,\wt\sH_V\ot^L\wt\sL_V)=\h^{q}(\wt V,\R\wt\i^!\R\wt\g_{V*}(\wt\g_V^*(\wt\sH_V)))=0.
\end{equation*}
\end{proof}

\begin{proposition}\label{intersection empty}
If $T\cap S=\emptyset$, the map $\lambda$ \eqref{lambda focus} is injective.
\end{proposition}
\begin{proof}
Since $\sH_0$ is locally constant, by (\cite{fu} 6.5.5), for any integer $q$,
\begin{equation*}
\h^q_U(\uu,\sH\ot^L\R\g_{U*}(\L))\isora \h^q_U(\uu,\R\gamma_{U*}\gamma^*_{U}(\sH)) \isora\h^q(U,\R\d^!_U\R\gamma_{U*}\gamma^*_{U}(\sH))=0.
\end{equation*}
Hence, we have a canonical isomorphism
\begin{equation*}
  \h^0_{D^{(R)}}(\xxr,\sHr\ot^L\sL^{(R)})\isora  \h^0_{X^{(R)}}(\xxr,\sHr\ot^L\sL^{(R)}).
\end{equation*}
Since $T\cap S=\emptyset$, we have $T^{(R)}\cap S^{(R)}=\emptyset$. Hence, for any object $\sG$ of $\dctf(\xxr, \L)$,
\begin{equation*}
 \h^0_{D^{(R)}}(\xxr,\sG)=\h^0_{S^{(R)}}(\xxr,\sG)\oplus\h^0_{ T^{(R)}}(\xxr,\sG).
\end{equation*}
In particular, the canonical map
\begin{equation*}
  \h^0_{S^{(R)}}(\xxr,\sHr\ot^L\sL^{(R)})\ra \h^0_{D^{(R)}}(\xxr,\sHr\ot^L\sL^{(R)})
\end{equation*}
is injective. Hence $\lambda$ is injective.
\end{proof}

\subsection{}
The pull-back by $\d^{(R)}$ gives a map
\begin{equation*}
  \h^0_{\sr}(\xxr,\sHr\ot^L\sL^{(R)})\ra \h^0_S(X,\d^{(R)*}(\sHr)\ot^L\d^{\dagger*}(\sL^{\dagger})).
\end{equation*}
Since the conductor of $\sF$ is $R$ and $\sF$ is isoclinic and clean along $D$ (\ref{notation cul}), the ramification of $\sF$ along $D$ is bounded by $R+$ (\ref{clean isocline sheaf}). Hence, we have an isomorphism \eqref{sheaf bound}
\begin{equation*}
\d^{(R)*}(j^{(R)}_*(\sH_0))\isora j_*(\d_U^*(\sH_0)).
\end{equation*}
We have an evaluation map
 \begin{equation*}
\d^{(R)*}(\sHr)\isora j_*(\d^*_U(\sH_0))(d)[2d]=j_*(\send(\sF))(d)[2d]\ra (j_*(\L_U))(d)[2d]=\L_X(d)[2d]=\sK_X,
\end{equation*}
where the third arrow is the push-forward of the trace map $\mathrm{Tr}:\send(\sF)\ra \L_U$. It induces a map
\begin{equation*}
  \ev^{(R)}:\h^0_S(X,\d^{(R)*}(\sHr)\ot^L\d^{\dagger*}(\sL^{\dagger}))
  \ra\h^0_S(X,\sK_X\ot^L\d^{\dagger*}(\sL^{\dagger}))\isora\h^0_S(X,\sK_X),
\end{equation*}
where second arrow is the inverse of the isomorphism \eqref{isoadjdg}.

\subsection{}
By \eqref{locdiag}, we have a map
\begin{equation*}
  \h^0(X,\d^{(R)*}j^{(R)}_*(\sH_0))\ti\h^0_S(X,\d^{\dagger*}(\sL^{\dagger})(d)[2d])\xra{\cup_S}
  \h^0_S(X,\d^{(R)*}(\sHr)\ot^L\d^{\dagger*}(\sL^{\dagger})).
\end{equation*}
In the following of this section, we denote by $e\in \h^0(X,\d^{(R)*}j^{(R)}_*(\sH_0))$ the unique pre-image of
$\id_{\sF}\in \End(\sF)=\h^0(X,j_*\d^*(\sH_0))$. The following diagram is commutative
\begin{equation}\label{right down square}
  \xymatrix{\relax
  \h^0_S(X,\d^{(R)*}(\sHr)\ot^L\d^{\dagger*}(\sL^{\dagger}))\ar[r]^-{\ev^{(R)}}&
  \h^0_S(X,\sK_X)\ar[d]^{\eqref{isoadjdg}}\\
  \h^0_S(X,\d^{\dagger *}(\sL^{\dagger})(d)[2d])\ar[u]_{e\cup_S-}\ar[r]^-{\cdot\rk_{\L}(\sF)}& \h^0_S(X,\sK_X\ot^L\d^{\dagger*}(\sL^{\dagger}))}
\end{equation}
since the composition of the following morphisms
\begin{equation*}
\h^0(X,\d^{(R)*}j^{(R)}_*(\sH_0))\ra \h^0(X,j_*\d_U^*(\sH_0))\xra{\rm{ev}} \h^0(X,\L)
\end{equation*}
maps $e\in \h^0(X,\d^{(R)*}j^{(R)}_*(\sH_0))$ to $\rk_{\L}(\sF)\in \L=\h^0(X,\L)$.

\subsection{}
We have a commutative diagram with Cartesian squares (\ref{notllcc})
\begin{equation}\label{diagR}
  \xymatrix{\relax
  \vvlp\ar@{}[rd]|{\Box}\ar[d]_-{\jr_V}\ar@{=}[r]&\vvlp\ar@{}[rrrd]|{\Box}
  \ar[d]\ar[r]^-{\varphi_2}&V\rtimes_kV\ar[r]^-{\varphi_1}&\vvd\ar[r]^-{j^{\dagger}_2}&
  (V\ti_kX)^{\dagger}\ar[d]^{j^{\dagger}_1}\\
  \xxr\ar@/_5mm/[rrrr]_{\varphi^{(R)}}\ar[r]&\xxlp\ar[rrr]^-{\varphi}& & & \xxd}
\end{equation}
 The base change maps give by composition the following isomorphism
\begin{equation*}
  \varphi^{(R)*}(\ol\sH^{\dagger})=\varphi^{(R)*}j^{\dagger}_{1!}\R j^{\dagger}_{2*}(\ol\sH^{\dagger}_V)\isora \jr_{V!}(j^{\dagger}_2\circ\varphi_1\circ\varphi_2)^*\R j^{\dagger}_{2*}(\ol\sH_V^{\dagger})\isora\jr_{V!}\wt\sH_V,
\end{equation*}
which induces a map
\begin{equation}\label{xxptoxxr}
\varphi^{(R)*}(\ol\sH^{\dagger})\isora \jr_{V !}(\wt\sH_V)\ra \jr_{V *}(\wt\nu_*(\sH_0))(d)[2d]=\sHr.
\end{equation}
By (\ref{jussmapR}), the composed map
\begin{equation*}
f^{(R)*}(\ol\sH)=\varphi^{(R)*}(f^*(\ol\sH))\xra{\eqref{defllcckey}} \varphi^{(R)*}(\ol\sH^{\dagger})\xra{\eqref{xxptoxxr}}\sHr
\end{equation*}
is equal to \eqref{fR}. We deduce by pull-back a commutative diagram
\begin{equation}\label{triangle cf}
\xymatrix{\relax
\h^0_X(\xx,\ol\sH)\ar[r]^-{\eqref{fupstar}}\ar[rd]_{\eqref{fRupstar}}&\h^0_{X^{\dagger}}((\xx)^{\dagger},\ol\sH^{\dagger})\ar[d]\\
&\h^0_{X^{(R)}}(\xxr,\sHr)}
\end{equation}

\subsection{}
We have the following diagrams with commutative squares
\begin{equation}\label{maindiagram3row}
  \xymatrix{\relax
  \h^0_{X^{\dagger}}((\xx)^{\dagger},\ol\sH^{\dagger})\ar[d]\ar[r]^-{\eqref{adjg'}}&
  \h^0_{X^{\dagger}}((\xx)^{\dagger},\ol\sH^{\dagger}\ot^L\sL^{\dagger})\ar[d]\\
  \h^0_{\xr}(\xxr, \sHr)\ar[r]^-{\eqref{HR ot adjunction}}&\h^0_{\xr}(\xxr, \sHr\ot^L\sL^{(R)})\\
  \h^{0}_X(\xxr,\L(d)[2d])\ar[u]^{e\cup-}\ar[r]^-{\theta}&
  \h^{0}_X(\xxr,\sL^{(R)}(d)[2d])\ar[u]_{e\cup-}}
\end{equation}
\begin{equation}\label{maindiagram3row2}
  \xymatrix{\relax
  \h^0_{X^{\dagger}}((\xx)^{\dagger},\ol\sH^{\dagger}\ot^L\sL^{\dagger})\ar[d]&  \h^0_{S^{\dagger}}((\xx)^{\dagger},\ol\sH^{\dagger}\ot^L\sL^{\dagger})     \ar@{_(->}[l]_-{\eqref{injkey}}\ar[d] \\
  \h^0_{\xr}(\xxr, \sHr\ot^L\sL^{(R)})\ar@{}[rd]|{(1)}&     \h^0_{\sr}(\xxr, \sHr\ot^L\sL^{(R)})     \ar[l]_-{\lambda}   \\
  \h^{0}_X(\xxr,\sL^{(R)}(d)[2d])\ar[u]_{e\cup-}&   \h^{0}_S(\xxr,\sL^{(R)}(d)[2d])\ar[u]^{e\cup_S-}  \ar[l]_-{\lambda_0}  \\     }
\end{equation}
\begin{equation}\label{maindiagram3row3}
  \xymatrix{\relax
 \h^0_{S^{\dagger}}((\xx)^{\dagger},\ol\sH^{\dagger}\ot^L\sL^{\dagger})\ar[d]\ar[r]^-{\eqref{pullback d'}}&\h^0_S(X,\d^{\dagger*}(\ol\sH^{\dagger})\ot^L\d^{\dagger*}(\sL^{\dagger}))\ar[r]^-{\eqref{ev'}}\ar[d] &\h^0_S(X,\sK_X)\ar@{=}[d] \\
  \h^0_{\sr}(\xxr, \sHr\ot^L\sL^{(R)})\ar[r]&
  \h^0_S(X,\d^{(R)*}(\sHr)\ot^L\d^{\dagger*}(\sL^{\dagger}))\ar[r]^-{\ev^{(R)}}\ar@{}[rd]|{\eqref{right down square}} &\h^0_S(X,\sK_X)\ar@{=}[d]  \\
  \h^{0}_S(\xxr,\sL^{(R)}(d)[2d])\ar[u]^{e\cup_S-}\ar[r]&
  \h^{0}_S(X,\d^{\dagger*}(\sL^{\dagger})(d)[2d])\ar[r]^-{\cdot\rk_{\L}(\sF)}\ar[u]_{e\cup_S-} & \h^0_S(X,\sK_X)}
\end{equation}
where
\begin{itemize}
  \item[i.]
  The arrows from the upper row to the middle row are the pull-backs by $\varphi^{(R)}$;

  \item[ii.]
  The arrows $e\cup-$ are the cup products \eqref{cap prod} and the arrows $e\cup_S-$ are given in \eqref{locdiag};

  \item[iii.]
  The arrows $\theta$ are induced by the canonical map $\L\ra \sL^{(R)}$;

  \item[iv.]
  The arrows $\lambda_0$ is the canonical map induced by the injection $S\ra X$;

\item[v.]
The square $(1)$ is commutative by \eqref{locdiag1};

\item[vi.]
  The arrows under \eqref{pullback d'} are the pull-back by $\d^{(R)}$;

\item[vii.]
We use the canonical isomorphism $\h^0_S(X,\sK_X)\isora\h^0_S(X,\sK_X\ot^L\d^{\dagger*}(\sL^{\dagger}))$ (cf. \ref{isoadjunction}).
\end{itemize}

\begin{lemma}\label{iso no h}
For any integer $q$, the canonical maps
\begin{equation*}
H^q_{S}(\xxr,\sL^{(R)}(d))\ra H^q_{X}(\xxr,\sL^{(R)}(d)),
\end{equation*}
are isomorphisms. In particular, the map $\lambda_0$ in \eqref{maindiagram3row2} is an isomorphism.
\end{lemma}
\begin{proof}
It is sufficient to show that, for any integer $q$,
\begin{equation*}
\h^q_{V}(\vvlp,\R \wt \gamma_{V*}\L(d))=0,
\end{equation*}
which follows from the fact that $\R\wt\d_V^! \R \wt \gamma_{V*}\L(d)=0$ \eqref{low*up!}.
\end{proof}

\subsection{}
For any integer $q$, any object $\sG$ of $\dctf(\xxr,\Lambda)$ and any closed subscheme $Z\in\xxr$, we denote by
\begin{equation*}
\h^q_{Z}(\xxr, \sG)\ra \h^q_{Z}(\xxr, \sG\ot^L \sL^{(R)}),\ \ \ x\mapsto x_a,
\end{equation*}
the morphism induced by the canonical map $\L\ra\sL^{(R)}$. For any closed immersion $Z\ra Y$ of closed subschemes of $\xxr$, we denote abusively by
\begin{equation*}
  \h^q_Z(\xxr,\sG)\ra\h^q_Y(\xxr,\sG),\ \ \ x\mapsto x,
\end{equation*}
the canonical map.

\begin{proposition}[\cite{saito cc} 3.3, 3.4]\label{id=e cup x}
We denote by $[X]\in \h^{0}_X(\xxr,\L(d)[2d])$ the cycle class of $\d^{(R)}(X)$. Then we have {\rm (\ref{sLdagger})}, \eqref{fRupstar}
\begin{equation*}
  f^{(R)*}(\id_{j_!(\sF)})=e\cup[X]\in \h^0_{\xr}(\xxr, \sHr).
\end{equation*}
\end{proposition}
The proof in (\cite{saito cc} 3.4) should be modified as in \ref{saito cc3.4}.

\begin{proposition}\label{theorem need 1}
If $T\cap S=\emptyset$ or if $\rk_{\L}(\sF)=1$, we have {\rm(\ref{defllcc})}
\begin{equation}\label{llcc=cycle}
  C_S(j_!(\sF))=\rk_{\L}(\sF)\cdot\d^{(R)*}(\lambda^{-1}_0([X]_a))\in \h^0_S(X,\sK_X).
\end{equation}
\end{proposition}
\begin{proof}
By \eqref{triangle cf}, \eqref{maindiagram3row} and \ref{id=e cup x}, we have (\ref{sLdagger})
\begin{equation*}
  \varphi^{(R)*}(\alpha(\id_{j_!(\sF)}))=e\cup([X]_a)\in\h^0_{\xr}(\xxr, \sHr\ot^L\sL^{(R)}).
\end{equation*}
Then, by \ref{lambda rank 1 iso}, \ref{intersection empty}, \ref{iso no h} and \eqref{maindiagram3row2}, we have \eqref{alpha0}
\begin{equation*}
  \varphi^{(R)*}(\a_0(j_!(\sF)))=e\cup_S(\lambda_0^{-1}([X]_a))\in\h^0_{\sr}(\xxr, \sHr\ot^L\sL^{(R)}).
\end{equation*}
Equation \eqref{llcc=cycle} follows form \eqref{maindiagram3row3}.
\end{proof}

\begin{lemma}\label{ts1 3.7}
We put $\wt X^{(R)}=\varphi^{(R)-1}(\d^{\dagger}(X))$ and $\wt S^{(R)}=\varphi^{(R)-1}(\d^{\dagger}(S))$ \eqref{diagR}. Then
\begin{itemize}
  \item[(i)]
  There exists a unique element $\tau\in \CH_d(\wt S^{(R)})$ which maps to $[X]-\varphi^{(R)!}[X]\in\CH_d(\wt X^{(R)})$, and we have
  \begin{equation}\label{cycle eq 1}
    \d^{(R)!}(\tau)=(-1)^d\cdot c_d\left(\O^1_{X/k}(\log D)\ot_{\sO_X}\sO_X(R)-\O^1_{X/k}(\log T)\right)^X_S\cap[X]\in\CH_0(S).
  \end{equation}
  \item[(ii)]
  We consider $\tau$ as an element in $\h^0_{\wt S^{(R)}}(\xxr,\L(d)[2d])$ by the cycle map {\rm (\ref{cycle map def})}. We have
  \begin{equation}\label{cycle eq 2}
  \tau_a=\lambda_0^{-1}([X]_a)\in \h^0_{\wt S^{(R)}}(\xxr,\sL^{(R)}(d)[2d]).
  \end{equation}
  \end{itemize}
\end{lemma}

The proof of this lemma is similar to that of (\cite{ts1} 3.7), in which the author consider the case where $\mathrm{supp}(R)=S$. It is an immediate application of (\cite{ksann} 3.4.9).

\begin{corollary}\label{theorem need 2}
We have
\begin{equation}\label{th need form 2}
\d^{(R)*}(\lambda^{-1}_0([X]_a))=(-1)^d\cdot c_d\left(\O^1_{X/k}(\log D)\ot_{\sO_X}\sO_{X}(R)-\O^1_{X/k}(\log T)\right)^X_S\cap[X]\in\h^0_S(X,\sK_X),
\end{equation}
where the right hand side is considered as an element of $\h^0_S(X,\sK_X)$ by the cycle map.
\end{corollary}
\begin{proof}
Applying \eqref{up*gysin!} to the map $\d^{(R)}:X\ra\xxr$, we have
\begin{equation}\label{cycle eq 3}
  \d^{(R)!}(\tau)=\d^{(R)*}(\tau)\in \h^0_S(X,\sK_X),
\end{equation}
where we consider $\d^{(R)!}(\tau)$ as an element of $\h^0_S(X,\sK_X)$ by the cycle map.
Since the following diagram
\begin{equation*}
  \xymatrix{\relax
  \h^0_{\wt S^{(R)}}(\xxr,\L(d)[2d])\ar[d]\ar[r]&\h^0_S(X,\sK_X)\ar[d]^{\eqref{isoadjdg}}\\
  \h^0_{\wt S^{(R)}}(\xxr,\sL^{(R)}(d)[2d])\ar[r]&\h^0_S(X,\sK_X\ot^L\d^{\dagger*}\sL^{\dagger})}
\end{equation*}
is commutative, where the horizontal arrows are the pull-backs by $\d^{(R)}$, we have
\begin{equation}\label{cycle eq 4}
 \d^{(R)*}(\tau)=\d^{(R)*}(\tau_a)\in \h^0_S(X,\sK_X).
\end{equation}
Hence, \eqref{th need form 2} follows form \eqref{cycle eq 1}, \eqref{cycle eq 2}, \eqref{cycle eq 3} and \eqref{cycle eq 4}.
\end{proof}

\subsection{}\label{proof keytheorem}
{\em Proof of Theorem {\rm\ref{keytheorem llcc lchern}}.}
By \ref{theorem need 1} and \ref{theorem need 2}, we have
\begin{eqnarray*}
  C_S(j_!(\sF))&=&(-1)^d\rk_{\L}(\sF)\cdot c_d\left(\O^1_{X/k}(\log D)\ot_{\sO_X}\sO_{X}(R)-\O^1_{X/k}(\log T)\right)^X_S\cap[X],\\
C_S(j_!(\L_U))&=&(-1)^d\cdot c_d\left(\O^1_{X/k}(\log D)-\O^1_{X/k}(\log T)\right)^X_S\cap[X]
\end{eqnarray*}
in $\h^0_S(X,\sK_X)$.
Hence,
\begin{eqnarray*}
  &&C_{S}(j_!(\sF))-\rk_{\L}(\sF)\cdot C_{S}(j_!(\L_U))\\
  &=&(-1)^{d}\rk_{\L}(\sF)\cdot c_d\left(\OX(\log D)\ot_{\sO_X}\sO_X(R)-\OX(\log D)\right)^X_{S}\cap[X]\in H^0_{S}(X,\sK_X).\nonumber
\end{eqnarray*}
\hfill$\Box$

\begin{remark}[\cite{as} 4.2.1, \cite{ts1} remark after 3.9]\label{lchern chern}
Observe that we have
\begin{eqnarray}
&& (-1)^{d}\cdot c_d\left(\O^1_{X/k}(\log D)\ot_{\sO_X}\sO_{X}(R)-\O^1_{X/k}(\log D)\right)^X_S\cap[X]\label{lchern=chern}\\
&=&-\{c(\O^1_{X/k}(\log D)^{\vee})\cap(1+c_1(\sO_X(R)))^{-1}\cap [R]\}_{\dim 0}\nonumber\\
&=&(-1)^d\cdot\{c(\O^1_{X/k}(\log D))\cap(1-c_1(\sO_X(R)))^{-1}\cap [R]\}_{\dim 0}\in \CH_0(S).\nonumber
\end{eqnarray}
\end{remark}

\subsection{}\label{simp cc}
we denote by
\begin{equation*}
\mathrm{T}^*X(\log D)=\mathbf{V}(\O^1_{X/k}(\log D)^{\vee}),
\end{equation*}
the logarithmic cotangent bundle of $X$. Since the $R$ is supported in $S$, $R=\sum_{i\in I_{\rw}}r_iD_i$, where $r_i\in\mathbb{Z}_{\ddy 0}$. By (\cite{saito cc} 3.16), For $i\in I_{\rw}$, we have \eqref{CCi}
\begin{equation*}
  CC_i(\sF)=r_i\cdot\rk_{\L}(\sF)\cdot\{c(\O^1_{X/k}(\log D))\cap (1-c_1(\sO_X(R)))^{-1}\cap [\mathrm{T}^*X(\log D)\ti_XD_i]\}_{\dim d}
\end{equation*}
in  $\CH_d(\mathrm{T}^*X(\log D)\ti_XD_i)$.
Hence, we have \eqref{CC0}
\begin{equation}
 CC^*(\sF)=\rk_{\L}(\sF)\cdot\{c(\O^1_{X/k}(\log D))\cap (1-c_1(\sO_X(R)))^{-1}\cap [\mathrm{T}^*X(\log D)\ti_X R]\}_{\dim d}
\end{equation}
in $\CH_d(\mathrm{T}^*X(\log D)\ti_X S)$.

\subsection{}
In the following, we take the notation and assumptions of {\rm\ref{X to Y not}}, and we assume that $Y$ is of dimension $1$ and that $Z$ is a closed point $y$ of $Y$, and we denote by $\ol y$ a geometric point localized at $y$, by $Y_{(\ol y)}$ the strict localization of $Y$ at $\ol y$ and by $\ol\eta$ a geometric generic point of $Y_{(\ol y)}$.

\begin{theorem}\label{pre-cond form}
 We assume that $S=D$ (i.e., $T=\emptyset$) or that $\rk_{\L}(\sF)=1$. Then, for any section $s:X\ra \mathrm{T}^*X(\log D)$, we have
\begin{equation}\label{cond form equ}
\mathrm{sw}_y(\R\G_c(U_{\ol\eta},\sF|_{U_{\ol\eta}}))
-\rk_{\L}(\sF)\cdot\mathrm{sw}_y(\R\G_c(U_{\ol\eta},\L))=(-1)^{d+1}\deg(CC^*(\sF)\cap[s(X)])
\end{equation}
in $\h^0_{\{y\}}(Y,\sK_Y)\isora\L$.
\end{theorem}
\begin{proof}
We denote by $\varpi:\mathrm{T}^*X(\log D)\ra X$ the canonical projection.
Since $\varpi\circ s=\id_X$, we have
\begin{eqnarray*}
  &&CC^*(\sF)\cap [s(X)]\\
  &=&\rk_{\L}(\sF)\cdot\{c(\O^1_{X/k}(\log D))\cap (1-c_1(\sO_X(R)))^{-1}\cap \varpi^*[R]\cap[s(X)]\}_{\dim=0}\\
  &=&\rk_{\L}(\sF)\cdot \varpi^*(\{c(\O^1_{X/k}(\log D))\cap (1-c_1(\sO_X(R)))^{-1}\cap [R]\}_{\dim=0})\cap[s(X)]\\
  &=&\rk_{\L}(\sF)\cdot \{c(\O^1_{X/k}(\log D))\cap (1-c_1(\sO_X(R)))^{-1}\cap [R]\}_{\dim=0}\in \CH_0(S).
\end{eqnarray*}
By \ref{keytheorem llcc lchern} and \eqref{lchern=chern}, we get (\ref{cycle map def})
\begin{equation*}
  (-1)^d(CC^*(\sF)\cap[s(X)])=C_{S}(j_!(\sF))-\rk_{\L}(\sF)\cdot C_{S}(j_!(\L_U))\in\h^0_S(X,\sK_X).
\end{equation*}
Hence, by (\ref{llcc ns cor}), we have
\begin{equation*}
  \mathrm{sw}_y(\R\G_c(U_{\ol\eta},\sF|_{U_{\ol\eta}}))
  -\rk_{\L}(\sF)
  \cdot\mathrm{sw}_y(\R\G_c(U_{\ol\eta},\L))=(-1)^{d+1}\pi_*(CC^*(\sF)\cap[s(X)])
\end{equation*}
in $\h^0_{\{y\}}(Y,\sK_Y)\isora\L$.
It is easy to see that the composed map
\begin{equation*}
  \CH_0(S)\xra{\cl}\h^0_S(X,\sK_X)\ra\h^0_{\{y\}}(Y,\sK_Y)\isora\L,
\end{equation*}
where the second arrow is the push-forward \eqref{proper pushf}, is just the degree map of zero cycles. We obtain \eqref{cond form equ}.
\end{proof}

\begin{remark}
Since $\pi:X\ra Y$ is proper and $T\cap V$ is a divisor with simple normal crossing relatively to $W$ (\ref{X to Y not}), the condition $S\cap T=\emptyset$ in \ref{keytheorem llcc lchern} implies $T=\emptyset$.
\end{remark}

\subsection{}\label{proof cond form}{\em Proof of Theorem {\rm\ref{cond form intro}}.}
Since $\sF$ is tamely ramified along $T\cap V$ relatively to $V$, $\sF|_{U_{\ol\eta}}$ is tamely ramified along $(T\cap V)_{\ol\eta}$ relatively to $V_{\ol\eta}$. By (\cite{illusie} 2.7, \cite{saito cc} 3.2), we have
\begin{eqnarray}
  \rk_{\L}(\R\G_c(U_{\ol\eta},\sF|_{U_{\ol\eta}}))&=&(-1)^{d-1}\rk_{\L}(\sF) \cdot c_{d-1}(\O^1_{V_{\ol\eta}/\ol\eta}(\log (T\cap V)_{\ol\eta}))\cap [V_{\ol\eta}]\nonumber\\
&=&\rk_{\L}(\sF)\cdot\rk_{\L}(\R\G_c(U_{\ol\eta},\L))\nonumber
\end{eqnarray}
in $\h^0(V_{\ol\eta},\sK_{V_{\ol\eta}})\isora\L$. Hence, we obtain \eqref{c f intro form} by \ref{pre-cond form}. $\hfill\Box$

\begin{remark}
We denote by $K$ the function field of $Y_{(\ol y)}$, by $\ol K$ a separable closure of $K$ and by $P$ the wild inertia subgroup of $\mathrm{Gal}(\ol K/K)$. We assume that $T=\emptyset$ and $Q$ is reduced. Notice that $X\ti_Y Y_{(\ol y)}$ is semi-stable over the strict trait $Y_{(\ol y)}$. Then the cohomology group
\begin{equation*}
  \h^*(U_{\ol\eta},\L)=\h^*_c(U_{\ol\eta},\L)
\end{equation*}
is tame, i.e., the action of $P$ is trivial (\cite{atml} 3.3). Hence, $\mathrm{sw}_y(\R\G_c(U_{\ol\eta},\L))=0$.
\end{remark}

\end{document}